\documentclass{article} 
\usepackage{iclr2027_conference,times}

\usepackage{amsmath,amsfonts,bm}

\def\eqref#1{equation~\ref{#1}}

\def\1{\bm{1}}

\DeclareMathAlphabet{\mathsfit}{\encodingdefault}{\sfdefault}{m}{sl}
\SetMathAlphabet{\mathsfit}{bold}{\encodingdefault}{\sfdefault}{bx}{n}

\usepackage{hyperref}
\usepackage{url}
\usepackage{amsthm}
\usepackage{algorithm}
\usepackage{algorithmic}
\usepackage{graphicx}
\usepackage{wrapfig}
\usepackage{mathrsfs}
\usepackage{booktabs}
\usepackage{multirow}
\usepackage{pifont}
\usepackage{listings}
\usepackage{lstautogobble}
\usepackage[scaled=0.92]{inconsolata}

\usepackage{xcolor}
\usepackage{colortbl}
\definecolor{lapandahighlight}{RGB}{255,238,238}

\lstdefinestyle{almCode}{
	basicstyle=\fontfamily{zi4}\selectfont\footnotesize,
	columns=fullflexible,
	keepspaces=true,
	tabsize=2,
	autogobble=true,
	showstringspaces=false,
	showtabs=false,
	showspaces=false,
	breaklines=true,
	breakatwhitespace=false,
	breakindent=0pt,
	breakautoindent=true,
	frame=single,
	framerule=0.35pt,
	rulecolor=\color{black!20},
	backgroundcolor=\color{black!2},
	xleftmargin=0.5em,
	xrightmargin=0.3em,
	framexleftmargin=0.3em,
	aboveskip=0.7em,
	belowskip=0.7em,
	captionpos=b,
	keywordstyle=\color{blue!60!black},
	commentstyle=\color{green!40!black},
	stringstyle=\color{red!50!black}
}

\lstdefinestyle{almPython}{
	style=almCode,
	language=Python,
	emph={SX,sym,CasadiProblem,build_solver,solve_alm,export_c_project},
	emphstyle=\color{blue!65!black}
}

\lstdefinestyle{almMatlab}{
	style=almCode,
	language=Matlab,
	emph={SX,sym,lapanda_create_solver,solve_alm},
	emphstyle=\color{blue!65!black}
}

\lstdefinestyle{almC}{
	style=almCode,
	language=C
}

\newcommand{\cmark}{\textcolor{green!60!black}{\ding{51}}}
\newcommand{\xmark}{\textcolor{red!75!black}{\ding{55}}}

\newtheorem{assumption}{Assumption}

\newtheorem{theorem}{Theorem}

\newtheorem{remark}{Remark}
\newcommand{\solver}[1]{\texttt{#1}}

\title{lapanda: A Matrix-Free Differentiable Solver for Nonconvex Constrained Optimization Layers}

\author{
\textbf{Yuankun Chen$^{1}$, Zifei Nie$^{1,*}$, Kangyu Lin$^{2}$,
J\'{a}n Drgo\v{n}a$^{3}$, Liang Wu$^{3}$} \\
\normalfont\normalsize
$^{1}$School of Artificial Intelligence, Jilin University, China \\
$^{2}$Graduate School of Informatics, Kyoto University, JPN \\
$^{3}$Department of Civil and Systems Engineering, Johns Hopkins University, USA
}

\makeatletter
\renewcommand{\@maketitle}{%
\vbox{%
\hsize\textwidth
\centering
\makebox[\textwidth][c]{%
\parbox{1.15\textwidth}{\centering\LARGE\bfseries \@title\par}%
}
\vskip 0.22in
{\large \@author\par}
\vskip 0.28in
}}
\makeatother

\iclrfinalcopy
\begin{document}

\maketitle
\lhead{}
\renewcommand{\headrulewidth}{0pt}
\begingroup
\renewcommand{\thefootnote}{*}
\footnotetext{Corresponding author (e-mail: zifei\_nie@jlu.edu.cn).}
\endgroup
\vspace{-4mm}
\begin{abstract}
Differentiable optimization brings the structural guarantees of mathematical optimization to network pipelines, allowing them to be trained end-to-end. 
However, its application remains challenging for nonconvex constrained problems, as existing differentiable solvers often suffer from limited modeling expressiveness due to their reliance on specialized problem structures, while also incurring substantial computation time and memory overhead in both the forward and backward passes.
To address these challenges, we propose \solver{lapanda}, a matrix-free differentiable solver for nonconvex optimization with general constraints.
It reformulates the problem to a sequence of augmented Lagrangian subproblems, each handled by a first-order inner solver through a proximal averaged quasi-Newton algorithm with adaptive linesearch, thus enabling efficient forward optimization. We establish local well-posedness of the solution map and convergence of the outer iterations, and further derive a sensitivity alignment between the original problem and the final subproblem in the backward pass, demonstrating that the subproblem sensitivity, which can be computed efficiently in a matrix-free manner, provides a principled approximation to the exact optimizer sensitivity.
We evaluate \solver{lapanda} on nonconvex constrained Rosenbrock benchmarks, imitation learning with several representative constrained optimal control problems, and embedded robotic obstacle-avoidance tasks. Compared with state-of-the-art differentiable solvers, \solver{lapanda} delivers substantial reductions in computation time and memory footprint while maintaining reliable constraint satisfaction and learning performance.

\noindent Code: \url{https://github.com/optiXlab1/lapanda}
\end{abstract}

\vspace{-8mm}
\section{Introduction}
\vspace{-3mm}

Differentiable optimization serves as a bridge between optimization and
machine learning, enabling structural priors and constraints to be embedded into learning frameworks~\citep{kotary2021endtoend}.
This paradigm has been applied to a wide range of learning and control problems and safety-critical applications, such as learning personalized driving behaviors from demonstrations~\citep{acerbo2022mpcil,acerbo2024drividoc}, integrating model-based planning into end-to-end autonomous driving systems~\citep{huang2024dtpp,huang2024dipp,karkus2023diffstack}, and combining  reinforcement learning with optimal control for robotic systems~\citep{romero2024actorcritic,wan2024difftori,xu2024leto}.

Despite its promise, differentiable optimization faces two practical
bottlenecks.
First, many existing differentiable solvers gain computational efficiency by tailoring their formulations to specific problem structures, including convexification~\citep{agrawal2019cvxpylayers}, least-squares objectives~\citep{pineda2022theseus}, and temporal structure in optimal control problems (OCPs)~\citep{amos2018differentiable}.
Second, both optimization and sensitivity computation typically rely on explicit KKT matrix construction and factorization, resulting in substantial computational and memory overhead~\citep{blondel2022efficient}. These bottlenecks are particularly restrictive for general nonconvex
constrained problems with limited resources.
We therefore propose \solver{lapanda}, a matrix-free differentiable solver for nonconvex optimization layers with general constraints. The contributions are summarized as follows:

\begin{itemize}
	\vspace{-2mm}
	\item \textbf{Efficient Optimization.} To enable efficient optimization under general constraints, \solver{lapanda} combines the augmented Lagrangian method (ALM) with a first-order \solver{PANDA}  solver. ALM handles the constraints, while \solver{PANDA} takes care of the resulting subproblems. We theoretically proved that the forward solution is locally well-posed and convergent.
	\item \textbf{Matrix-Free Differentiation.}
	To differentiate through the optimizer without constructing derivative matrices, we analyzed the alignment between the differential KKT systems of the original problem and the last ALM subproblem, then derived an error bound for approximating the exact solution sensitivity by the one of the final ALM subproblem. This leads to a backward pass based only on matrix-vector product operators.
	\item \textbf{Comprehensive Evaluation.}
	To evaluate \solver{lapanda} from multiple perspectives, we considered three
	classes of optimization problems with distinct structures: nonconvex constrained
	Rosenbrock benchmarks, imitation learning via parametric constrained OCP, and embedded mobile
	robot obstacle-avoidance tasks. The results show that \solver{lapanda} reliably completes the tasks while achieving competitive computational and memory efficiency.
	\item \textbf{Cross-Platform Software.}
	To facilitate practical deployment, we developed a C-based solver core with user-friendly Python and MATLAB interfaces, together with engineering optimizations that reduce cross-language interface overhead.

\end{itemize}
\vspace{-5mm}
\section{Related Work}
\vspace{-2mm}
\label{sec:related_work}

Since \solver{OptNet} introduced the idea of embedding a parametric optimization problem as a differentiable layer into a learning architecture~\citep{amos2017optnet}, a growing number of differentiable optimization solvers have been crafted. Many of these methods target specific problem classes. Examples include differentiable quadratic programming layers such as \solver{qpth} and ADMM-based QP layers~\citep{amos2017optnet,butler2023efficient}, disciplined convex and conic optimization layers such as \solver{CVXPYLayers} and \solver{diffcp}~\citep{agrawal2019cvxpylayers,agrawal2019diffcp}, and nonlinear least-squares layers such as \solver{Theseus}~\citep{pineda2022theseus}. Despite their effectiveness, these specialized formulations make them less suitable for general nonconvex constrained optimization problems.

More general differentiable optimization frameworks are also available. For general nonlinear programs, \solver{CasADi} combines interior-point optimization with active-set KKT differentiation~\citep{andersson2019casadi,andersson2018casadi_sensitivity}. For general constrained OCPs, \solver{acados} combines structure-exploiting SQP with smoothed interior-point sensitivities~\citep{verschueren2022acados,frey2025differentiable_nmpc}, while \solver{SafePDP} computes sensitivities through a differentiable Pontryagin method~\citep{jin2021safepdp}. More recently, \solver{TurboMPC} exploits GPU parallelism to accelerate differentiable model predictive control (MPC)~\citep{bravopalacios2026turbompc}. However, despite their broad applicability, these frameworks often explicitly construct and factorize matrices arising from KKT systems, imposing considerable computational and memory overhead, especially on resource-limited platforms.

A complementary line of work reduces this second-order burden. \solver{FFOLayer} approximates hypergradients using only first-order information for convex programs, but requires additional forward optimization~\citep{zhao2025ffolayer}. Building on the forward-backward envelope framework~\citep{stella2017simple,themelis2018forward}, \solver{PANDA} employs first-order iterations with adaptive stepsizes in the forward pass and matrix-free Krylov methods based on matrix-vector products in the backward pass, but supports only constraints that admit efficient proximal mappings~\citep{chen2026panda}.

Building on these ideas, \solver{lapanda} extends the \solver{PANDA} framework to general constraints. While prior work has combined ALM with proximal-type methods for forward optimization~\citep{pas2022alpaqa}, we further establish local well-posedness and convergence for this formulation and exploit the subproblem structure to derive the sensitivity alignment and efficient matrix-free differentiation. Table~\ref{tab:related_solver_comparison} highlights the unique positioning of \solver{lapanda} among representative differentiable solvers.
\vspace{-4mm}
\begin{table}[h]
	\centering
	\caption{Comparison of representative differentiable optimization solvers.}
	\label{tab:related_solver_comparison}
	
	\resizebox{0.92\columnwidth}{!}{
		\begin{tabular}{lcccc}
			\toprule
			Solver
			& Target problem
			& Hard constraints
			& Matrix-free
			& Platform \\
			\midrule
			
			\rowcolor{lapandahighlight}
			lapanda (Ours)
			& {\color{green!60!black}General NLP}
			& {\color{green!60!black}General}
			& \cmark
			& {\color{green!60!black}C / MATLAB / Python} \\
			
			CasADi (2019)
			& {\color{green!60!black}General NLP}
			& {\color{green!60!black}General}
			& \xmark
			& {\color{green!60!black}C / MATLAB / Python} \\
			
			acados (2025)
			& {\color{blue!70!black}General OCP}
			& {\color{green!60!black}General}
			& \xmark
			& {\color{green!60!black}C / MATLAB / Python} \\
			
			TurboMPC (2026)
			& {\color{blue!70!black}General OCP}
			& {\color{green!60!black}General}
			& \xmark
			& Python \\
			
			SafePDP (2021)
			& {\color{blue!70!black}General OCP}
			& {\color{green!60!black}General}
			& \xmark
			& Python \\
			
			PANDA (2026)
			& Composite NLP
			& Prox.-friendly
			& \cmark
			& MATLAB \\
			
			FFOLayer (2026)
			& Convex
			& Convex
			& \cmark
			& Python \\
			
			mpc.pytorch (2018)
			& OCP
			& Box
			& \xmark
			& Python \\
			
			Theseus (2022)
			& NLS
			& --
			& \xmark
			& Python \\
			
			qpth (2017)
			& Convex QP
			& Convex
			& \xmark
			& Python \\
			
			CVXPYLayers (2019)
			& Convex
			& Convex
			& \xmark
			& Python \\
			
			\bottomrule
		\end{tabular}
	}
\end{table}

\section{Problem Formulation}
\vspace{-2mm}
We consider a parametric constrained optimization problem of the form
\begin{equation}
	\begin{aligned}
		u^\star(\theta) \in
		\underset{u\in\mathbf{U}}{\operatorname{argmin}}
		\quad & \ell(u,\theta) \\
		\mathrm{s.t.}\quad
		& c(u,\theta)=0,\\
		& h(u,\theta)\le 0,
	\end{aligned}
	\tag{P}
	\label{eq:problem_p}
\end{equation}
where \(u\in\mathbb{R}^{n_u}\) is the decision variable and 
\(\theta\in\mathbb{R}^{n_\theta}\) denotes the learnable parameter.
The objective
\(\ell:\mathbb R^{n_u}\times\mathbb R^{n_\theta}\to\mathbb R\)
and the constraint functions
\(c:\mathbb R^{n_u}\times\mathbb R^{n_\theta}\to\mathbb R^{n_c}\) and
\(h:\mathbb R^{n_u}\times\mathbb R^{n_\theta}\to\mathbb R^{n_h}\)
are smooth and possibly nonconvex; \(c\) and \(h\) represent equality and
inequality constraints, respectively.
The set \(\mathbf U\) represents a simple constraint set, such as a box,
whose proximal mapping can be evaluated efficiently.

In differentiable optimization, the \textbf{\emph{forward pass}} corresponds to solving
(\ref{eq:problem_p}) for a given parameter \(\theta\), yielding an optimizer
\(u^\star(\theta)\). In many learning problems, \(\theta\) must also be adjusted
so that the resulting optimizer minimizes a smooth task-level objective \(\Phi(\theta)\). This leads to the outer learning problem
\begin{equation}
	\min_{\theta}
	\quad
		\Phi(\theta)
		:=
		\mathcal{L}_{\mathrm{out}}\big(u^\star(\theta),\theta\big),
	\label{eq:outer_problem}
\end{equation}
where \(\mathcal{L}_{\mathrm{out}}\) can be an imitation loss or performance loss.
Differentiating the outer objective yields the chain rule as follows
\begin{equation}
	\nabla_\theta \Phi(\theta)
	=
		\nabla_u \mathcal{L}_{\mathrm{out}}\big(u^\star(\theta),\theta\big)
	\frac{\partial u^\star(\theta)}{\partial \theta}
	+
		\nabla_\theta \mathcal{L}_{\mathrm{out}}\big(u^\star(\theta),\theta\big).
	\label{eq:outer_gradient_direct}
\end{equation}
The main computational challenge in evaluating Eq.~(\ref{eq:outer_gradient_direct}) lies in the optimizer sensitivity
$
\frac{\partial u^\star(\theta)}{\partial \theta},
$
whose efficient computation constitutes the \textbf{\emph{backward pass}} of (\ref{eq:problem_p}).
\vspace{-2mm}
\paragraph{Optimal control problems.}
Constrained optimal control problems are common in the control of autonomous systems.
A typical discretized finite-horizon OCP has the form
\begin{equation}
	\begin{aligned}
		\xi^\star_\theta
		=
		\{x^\star_{0:N},u^\star_{0:N-1}\}
		\in
		\underset{\{x_k,u_k\}}{\operatorname{argmin}}
		\quad &
		\sum_{k=0}^{N-1}
		\ell_k(x_k,u_k;\theta)
		+
		\ell_N(x_N;\theta)
		\\
		\mathrm{s.t.}\quad
		&
		x_0=x_{\mathrm{init}}(\theta),
		\\
		&
		x_{k+1}=f_k(x_k,u_k;\theta),
		\\
		&
		c_k(x_k,u_k;\theta)=0,
		\quad
		h_k(x_k,u_k;\theta)\le 0,
		\\
		&
		c_N(x_N;\theta)=0,
		\quad
		h_N(x_N;\theta)\le 0,
		\\
		&
		u_k\in\mathbf{U}_k,
		\quad k=0,\dots,N-1 .
	\end{aligned}
	\tag{OCP}
	\label{eq:ocp}
\end{equation}
Here, \(x_k\in\mathbb{R}^{n_x}\) and \(u_k\in\mathbb{R}^{n_u}\) denote the state and control, respectively, and
\(f_k:\mathbb{R}^{n_x}\times\mathbb{R}^{n_u}\times\mathbb{R}^{n_\theta}\to\mathbb{R}^{n_x}\)
denotes the dynamics. The functions \(\ell_k\), \(c_k\) and \(h_k\) are the stage-wise counterparts of
\(\ell\), \(c\) and \(h\) in problem~(\ref{eq:problem_p}),
with \(\ell_N\), \(c_N\), and \(h_N\) denoting their terminal counterparts.
While (\ref{eq:ocp}) naturally has a stage-wise dynamic structure, a single-shooting transcription~\citep{gros2024numerical} recursively eliminates the states as decision variables, yielding a control-only constrained optimization problem in the generic form~(\ref{eq:problem_p}). This allows \solver{lapanda} to be applied directly. We therefore develop the solver using (\ref{eq:problem_p}) and defer the transcription details to Appendix~A.

\vspace{-3mm}
\section{lapanda}
\vspace{-2mm}
\begin{wrapfigure}{r}{0.6\linewidth}
	\vspace{-6mm}
	\centering
	\includegraphics[
	width=0.9\linewidth,
	trim={0mm 0cm 0mm 0cm},
	clip
	]{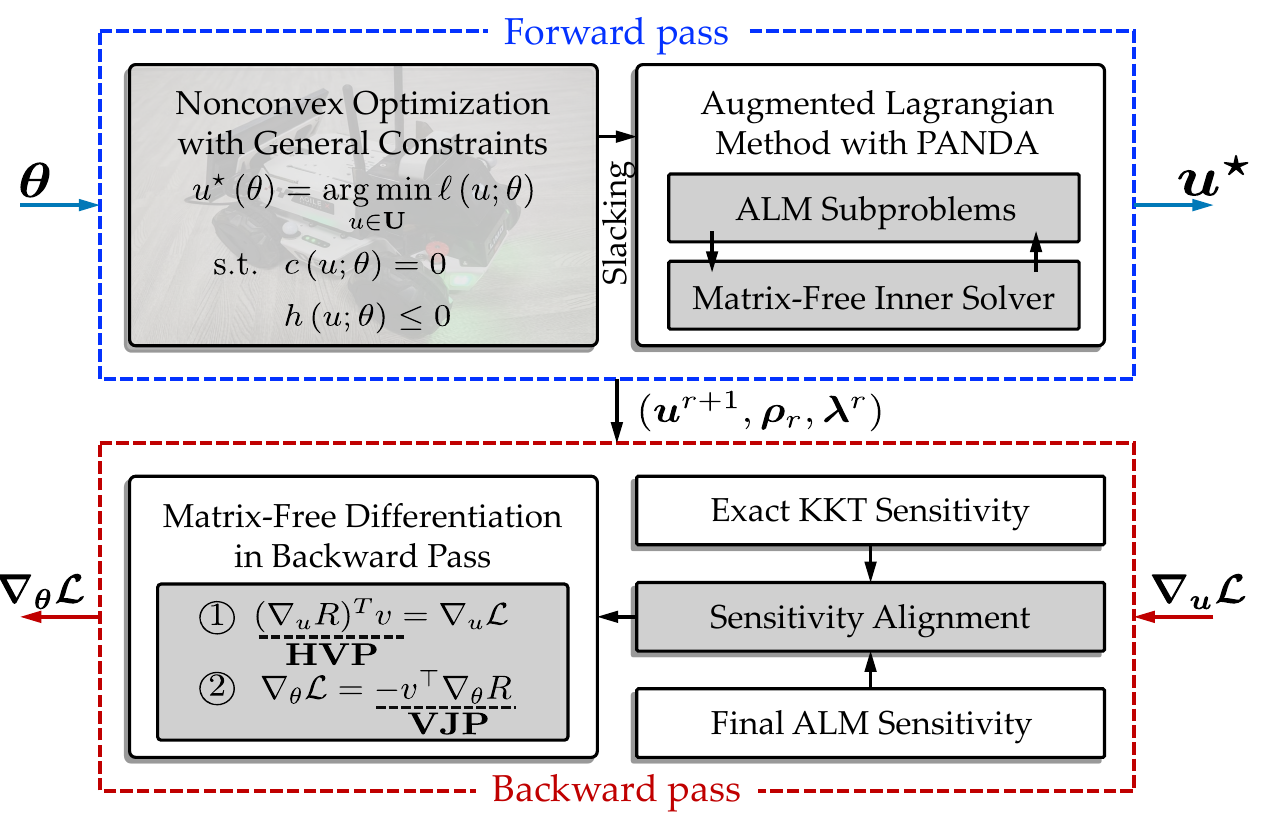}
	\vspace{-5mm}
	\caption{Overview of the lapanda framework.}
	\vspace{-10mm}
	\label{fig:overview}
\end{wrapfigure}

This section details the algorithmic principle of \solver{lapanda}, as shown in
Figure~\ref{fig:overview}. The \textbf{\emph{forward pass}} is detailed in
Section~\ref{subsec:forward_pass}, where the nonconvex constrained problem is reformulated
into ALM subproblems and reduced to a composite form handled by
\solver{PANDA}. In Section~\ref{subsec:backward_pass}, we derive the
\textbf{\emph{backward pass}} by comparing the sensitivity systems of the
original problem and the final ALM subproblem, thereby motivating a matrix-free
approximate implementation.

\subsection{Forward Pass: Augmented Lagrangian with PANDA}
\label{subsec:forward_pass}
\vspace{-2mm}
We now describe the forward pass of \solver{lapanda} for solving the
parameterized problem~(\ref{eq:problem_p}).
To make the following calculation and analysis well defined, we impose the
standing assumptions below.

\begin{assumption}
	\label{ass:local_regular}
	Let \(u^\star\) be a local solution of problem~(\ref{eq:problem_p}). 
	We assume that:
	\vspace{-2mm}
	\begin{itemize}
		\item[A1.] The functions $\ell(u,\theta)$, $c(u,\theta)$, and $h(u,\theta)$ are
		continuously differentiable and their derivatives are
		locally Lipschitz continuous. The simple set $\mathbf U$ is nonempty and closed.
		\item[A2.] The functions
		\(\ell(u,\theta)\), \(c(u,\theta)\), and \(h(u,\theta)\) are \(C^{2,1}\)
		(twice continuously differentiable with locally Lipschitz second derivatives) in a
		neighborhood of \(u^\star\). The set \(\mathbf U\) admits a local \(C^{2,1}\)
		active-constraint representation near \(u^\star\). With this representation,
		the full constrained problem satisfies the linear independence constraint qualification (LICQ), the second-order sufficient condition (SOSC), and strict complementarity at \(u^\star\).
	\end{itemize}
\end{assumption}

\vspace{-4mm}
\paragraph{Slack reformulation and augmented Lagrangian.}
Firstly, we rewrite the general equality and inequality constraints in a
unified set-membership form by defining
\[
G(u,\theta):=
\begin{bmatrix}
	c(u,\theta)\\
	h(u,\theta)
\end{bmatrix},
\qquad
\mathcal D:=\{0\}^{n_c}\times\mathbb R_-^{n_h}.
\]
Then the constraints in~(\ref{eq:problem_p}) can be compactly written as
\(G(u,\theta)\in\mathcal D\). By introducing an auxiliary slack variable
\(z\in\mathcal D\), problem~(\ref{eq:problem_p}) is equivalently written as
\begin{equation}
	\begin{aligned}
		\min_{u\in\mathbf U,\ z\in\mathcal D}
		\quad & \ell(u,\theta)\\
		\mathrm{s.t.}\quad
		& G(u,\theta)-z=0 .
	\end{aligned}
	\label{eq:p_alm}
\end{equation}
With a penalty parameter \(\rho>0\) and multiplier
\(\lambda\in\mathbb R^{n_c+n_h}\), its augmented
Lagrangian is given as
\begin{equation}
	\mathcal L_{\rho}(u,z,\lambda;\theta)
	=
	\ell(u,\theta)
	+
	\lambda^\top\big(G(u,\theta)-z\big)
	+
	\frac{\rho}{2}
	\|G(u,\theta)-z\|^2 .
	\label{eq:alm_lagrangian}
\end{equation}
At the \(r\)-th ALM iteration, given \((\lambda_r,\rho_r)\), the corresponding augmented Lagrangian subproblem reads
\begin{equation}
	\begin{aligned}
		(u_{r+1},z_{r+1})
		\approx
		\underset{u\in\mathbf{U},\;z\in\mathcal D}{\operatorname{argmin}}
		\mathcal L_{\rho_r}(u,z,\lambda_r;\theta).
	\end{aligned}
	\tag{P-sub}
	\label{eq:alm_subproblem}
\end{equation}
Then, the ALM forward pass proceeds as follows: (i) Solve~(\ref{eq:alm_subproblem}); (ii) Update the multiplier using the resulting feasibility residual; and (iii) Increase the penalty parameter only when the residual fails to decrease sufficiently. The complete calculation steps are summarized in Algorithm~\ref{alg:alm_forward}.

\begin{algorithm}[t]
	\caption{Forward pass of lapanda}
	\label{alg:alm_forward}
	\begin{algorithmic}[1]
		\STATE \textbf{Input:} parameter \(\theta\), initial primal variable \(u_0\),
		multiplier \(\lambda_0\), penalty \(\rho_0\), constants
		\(\kappa_{\mathrm{dec}}\in(0,1)\), \(\tau_\rho>1\), tolerances
		\(\varepsilon_{\mathrm{ALM}}>0\), and maximum number of ALM iterations
		\(R_{\max}\in\mathbb N_+\)
		\STATE Set \(z_0=\Pi_{\mathcal D}(G(u_0,\theta)+\rho_0^{-1}\lambda_0)\) and \(s_0=G(u_0,\theta)-z_0\)
		\FOR{\(r=0,1,\dots,R_{\max}-1\)}
		\STATE Solve the ALM subproblem
		\(
		(u_{r+1},z_{r+1})\approx
		\arg\min_{u\in\mathbf U,\ z\in\mathcal D}
		\mathcal L_{\rho_r}(u,z,\lambda_r;\theta)
		\)
		\STATE Evaluate \(s_{r+1}=G(u_{r+1},\theta)-z_{r+1}\)
		\STATE Update \(\lambda_{r+1}=\lambda_r+\rho_r s_{r+1}\)
		\STATE Set \(\rho_{r+1}=\rho_r\) if \(\|s_{r+1}\|\le \kappa_{\mathrm{dec}}\|s_r\|\), and \(\rho_{r+1}=\tau_\rho\rho_r\) otherwise
		\IF{\(\|s_{r+1}\|_\infty
			\leq \varepsilon_{\mathrm{ALM}}\)}
		\STATE \textbf{break}
		\ENDIF
		\ENDFOR
		\STATE \textbf{Return:} \(u_{r+1}\approx u^\star(\theta)\)
	\end{algorithmic}
\end{algorithm}

\begin{remark}[Inner subproblem solved by PANDA]
	\label{rem:inner_subproblem}
	The main computational burden in Algorithm~\ref{alg:alm_forward} arises from
	solving the ALM subproblem. Following the reduction in \citet{pas2022alpaqa},
	we eliminate \(z\) before invoking the inner solver, yielding a reduced problem only in \(u\)
	\begin{equation}
		u_{r+1}
		\approx
		\underset{u\in\mathbf U}{\operatorname{argmin}}
		\psi_{\rho_r}(u;\theta,\lambda_r),
		\label{eq:reduced_alm_subproblem}
	\end{equation}
	where
	\begin{equation}
		\psi_{\rho}(u;\theta,\lambda)
		:=
		\ell(u,\theta)
		+
		\frac{\rho}{2}
		\operatorname{dist}^2
		\left(
		G(u,\theta)+\rho^{-1}\lambda,
		\mathcal D
		\right).
		\label{eq:reduced_alm_objective}
	\end{equation}
	Since \(\mathbf U\) typically represents simple control-input constraints, often in the form of box constraints, this reduced problem is equivalently written as
	the smooth-plus-nonsmooth composite problem
	\begin{equation}
		\min_u
		\quad
		\psi_{\rho_r}(u;\theta,\lambda_r)
		+
		\delta_{\mathbf U}(u),
		\label{eq:panda_composite_subproblem}
	\end{equation}
	which can be handled efficiently in a matrix-free manner by \solver{PANDA}.
	The \(z\)-elimination derivation and the inner \solver{PANDA} procedure are
	given in Appendix~\ref{app:alm_reduction} and
	Appendix~\ref{app:panda_composite_subproblem}, respectively.
\end{remark}

\begin{remark}[Warm start]
	\label{rem:warm_starting}
	Warm-starting plays a key role in the computational efficiency of \solver{lapanda}.
	The initial primal variable, multiplier, and penalty parameter affect both the
	number of ALM outer iterations and the conditioning of the subproblem.
	An overly small penalty may require repeated penalty increases, whereas an overly large
	penalty can make the inner problem ill-conditioned. Thus, warm-starting with information from previous
	solves is an effective strategy.
\end{remark}
For the forward pass computation described above, we draw on local augmented
Lagrangian theory~\citep{bertsekas1999nonlinear,nocedal2006numerical} to
establish the following local well-posedness and convergence result.

\begin{theorem}[Local well-posedness and convergence]
	\label{thm:lapanda_wellposed}
	Suppose Assumption~\ref{ass:local_regular} holds at a local solution
	\((u^\star,z^\star,\lambda^\star)\) of the slack formulation Eq.~(\ref{eq:p_alm}). Then there exist
	positive constants \(\bar\rho\), \(\delta\), \(\epsilon\), and \(M\) such
	that, for any \(\lambda_r\) and \(\rho_r\) satisfying
	$\|\lambda_r-\lambda^\star\|\le \rho_r\delta$ and $\rho_r\ge\bar\rho$, 
	the following statements hold:
	\begin{itemize}
		\vspace{-2mm}
		\item[(i)] The slack ALM subproblem~(\ref{eq:alm_subproblem}) is locally well defined. More precisely, it admits a unique local solution \((u_{r+1},z_{r+1})\) in the \(\epsilon\)-neighborhood of \((u^\star,z^\star)\). Moreover, the subproblem satisfies LICQ and SOSC at \((u_{r+1},z_{r+1})\).
		
		\vspace{-2mm}
		\item[(ii)] With the multiplier update in
		Algorithm~\ref{alg:alm_forward}, the local primal-dual estimates
		satisfy
		\begin{equation}
			\|(u_{r+1},z_{r+1})-(u^\star,z^\star)\|
			\le
			M\frac{\|\lambda_r-\lambda^\star\|}{\rho_r},
			\qquad
			\|\lambda_{r+1}-\lambda^\star\|
			\le
			M\frac{\|\lambda_r-\lambda^\star\|}{\rho_r}.
			\label{eq:alm_local_estimates_main}
		\end{equation}
		
		\vspace{-2mm}
		\item[(iii)] After eliminating \(z\), the reduced composite
		subproblem (\ref{eq:panda_composite_subproblem}) satisfies the local
		\solver{PANDA} regularity conditions at \(u_{r+1}\). In particular,
		\(u_{r+1}\) is a strong local minimizer, and
		\(\delta_{\mathbf U}\) is locally prox-regular relative to the
		identified active manifold. Consequently, the \solver{PANDA} inner iteration
		is locally well posed and converges locally to \(u_{r+1}\).
	\end{itemize}
\end{theorem}
\vspace{-2mm}
\begin{proof}
	See Appendix~\ref{app:convergence}.
\end{proof}

\vspace{-4mm}
\subsection{Backward Pass: Sensitivity Alignment and Matrix-Free Differentiation}
\label{subsec:backward_pass}
\vspace{-2mm}
This subsection derives the backward pass of \solver{lapanda} through sensitivity alignment and matrix-free differentiation. As shown in the
outer-gradient expression Eq.~(\ref{eq:outer_gradient_direct}), the key
computational quantity is the solution sensitivity
\(D_\theta u^\star\). We first characterize this sensitivity through the
optimality conditions of the original constrained problem.
\vspace{-2mm}
\paragraph{Exact sensitivity from the original KKT system.}
Under Assumption~\ref{ass:local_regular}, strict complementarity implies local
active-set identification. Let \(\mathcal A_h\) be the active set of
\(h(u,\theta)\le0\) at \(u^\star\). We write the locally active smooth
constraints as
\begin{equation}
	a(u,\theta):=
	\begin{bmatrix}
		c(u,\theta)\\
		h_{\mathcal A_h}(u,\theta)
	\end{bmatrix},
	\qquad
	\phi(u)=0,
	\label{eq:active_manifold_constraints}
\end{equation}
where \(\phi(u)=0\) represents the identified active manifold of the simple set
\(\mathbf U\). 

Let \(A=\nabla_u a(u^\star,\theta)\),
\(C=\nabla_u\phi(u^\star)\), and let \(\eta^\star,\mu^\star\) be the
corresponding multipliers. The local active-set KKT system of Problem~(\ref{eq:problem_p}) is
\begin{equation}
	\nabla_u\ell(u^\star,\theta)+A^\top\eta^\star+C^\top\mu^\star=0,\qquad
	a(u^\star,\theta)=0,\qquad
	\phi(u^\star)=0.
		\label{eq:kkt_conditions}
\end{equation}
Applying the implicit function theorem to Eq.~(\ref{eq:kkt_conditions}) and differentiating with respect to \(\theta\) gives
\begin{equation}
	\label{eq:original_kkt_sensitivity}
	\begin{bmatrix}
		H & A^\top & C^\top\\
		A & 0 & 0\\
		C & 0 & 0
	\end{bmatrix}
	\begin{bmatrix}
		D_\theta u^\star\\
		D_\theta \eta^\star\\
		D_\theta \mu^\star
	\end{bmatrix}
	=
	-
	\begin{bmatrix}
		B\\
		D\\
		0
	\end{bmatrix},
\end{equation}
where
$
H:=\nabla_{uu}^2(\ell+\eta^{\star\top}a+\mu^{\star\top}\phi),\quad
B:=\nabla_{u\theta}^2(\ell+\eta^{\star\top}a+\mu^{\star\top}\phi),
$
and
$ D:=\nabla_\theta a. $

Eq.~(\ref{eq:original_kkt_sensitivity}) characterizes the exact sensitivity required by the backward pass. Instead of explicitly forming and solving this saddle-point system, which would incur substantial memory overhead, we align it with the optimality conditions of the last ALM subproblem. Such alignment exposes the same composite structure exploited by \solver{PANDA} in the forward solution pass, thereby enabling a matrix-free backward pass with substantially reduced memory requirements.

\paragraph{Sensitivity alignment through the final ALM subproblem.}
For the final ALM subproblem~(\ref{eq:alm_subproblem}), let
\(\eta_r\) denotes the components of the incoming ALM multiplier
\(\lambda_r\) associated with the equality and locally active inequality
constraints. Define the corresponding shifted multiplier as
$
\bar\eta_{r+1}
:=
\eta_r
+
\rho_r a(u_{r+1},\theta).
$
The local optimality conditions of the final reduced
\solver{lapanda} subproblem are
\begin{equation}
	\begin{aligned}
		\nabla_u\ell(u_{r+1},\theta)
		+A_{r+1}^\top\bar\eta_{r+1}
		+C_{r+1}^\top\mu_{r+1}
		=0,
		a(u_{r+1},\theta)
		-\rho_r^{-1}
		\left(\bar\eta_{r+1}-\eta_r\right)
		=0,
		\phi(u_{r+1})=0,
	\end{aligned}
	\label{eq:alm_kkt}
\end{equation}
where \(A_{r+1}:=\nabla_u a(u_{r+1},\theta)\) and
\(C_{r+1}:=\nabla_u\phi(u_{r+1})\).

When differentiating the final ALM subproblem, the incoming multiplier
\(\eta_r\) and penalty parameter \(\rho_r\) are treated as fixed.
Differentiating Eq.~(\ref{eq:alm_kkt}) therefore gives
\begin{equation}
	\label{eq:lapanda_sensitivity}
	\begin{bmatrix}
		H_{r+1} & A_{r+1}^\top & C_{r+1}^\top\\
		A_{r+1} & -\rho_r^{-1}I & 0\\
		C_{r+1} & 0 & 0
	\end{bmatrix}
	\begin{bmatrix}
		D_\theta u_{r+1}\\
		D_\theta\bar\eta_{r+1}\\
		D_\theta\mu_{r+1}
	\end{bmatrix}
	=
	-
	\begin{bmatrix}
		B_{r+1}\\
		D_{r+1}\\
		0
	\end{bmatrix},
\end{equation}
where
$
H_{r+1}
:=
\nabla_{uu}^2
\left(
\ell+\bar\eta_{r+1}^\top a+\mu_{r+1}^\top\phi
\right),
B_{r+1}
:=
\nabla_{u\theta}^2
\left(
\ell+\bar\eta_{r+1}^\top a+\mu_{r+1}^\top\phi
\right)
$
and
$
D_{r+1}:=\nabla_\theta a,
$
with all derivatives evaluated at \(u_{r+1}\).
The detailed derivation of Eqs.~(\ref{eq:alm_kkt})--(\ref{eq:lapanda_sensitivity})
is provided in Appendix~\ref{app:alm_sensitivity_derivation}.
Comparing Eq.~(\ref{eq:lapanda_sensitivity}) with
Eq.~(\ref{eq:original_kkt_sensitivity}), the two systems differ only through the ALM regularization block \(-\rho_r^{-1}I\) and the finite-iterate mismatch from the original primal-dual solution. This motivates the following theory.

\begin{theorem}[Sensitivity alignment]
	\label{thm:sensitivity_alignment}
	Suppose the assumptions of
	Theorem~\ref{thm:lapanda_wellposed} hold. Define
	\[
	\epsilon_{r+1}
	:=
	\|u_{r+1}-u^\star\|
	+
	\|\bar\eta_{r+1}-\eta^\star\|
	+
	\|\mu_{r+1}-\mu^\star\|.
	\]
	Then the sensitivity obtained from Eq.~(\ref{eq:lapanda_sensitivity}) satisfies
	\begin{equation}
		D_\theta u_{r+1}
		=
		D_\theta u^\star
		+
		O(\rho_r^{-1})
		+
		O(\epsilon_{r+1}).
		\label{eq:sensitivity_error}
	\end{equation}
	If the ALM outer iterates converge to the local primal-dual solution, then
	\(\epsilon_{r+1}\to0\), and the remaining discrepancy is
	\(O(\rho_r^{-1})\). Moreover, if \(\rho_r\to\infty\), then
	\(D_\theta u_{r+1}\to D_\theta u^\star\).
\end{theorem}
\begin{proof}
	See Appendix~\ref{app:proof_sensitivity_alignment}.
\end{proof}
\vspace{-2mm}
Theorem~\ref{thm:sensitivity_alignment} explains that the sensitivity obtained by differentiating the converged ALM subproblem approximates the exact solution sensitivity with an explicit error bound. Under the residual viewpoint, without explicitly constructing the system Eq.~(\ref{eq:lapanda_sensitivity}), we compute \(D_\theta u_{r+1}\) through the corresponding adjoint system using matrix-free Krylov methods and automatic-differentiation-based operators. Moreover, the penalty term in the subproblem yields locally positive curvature in the reduced space, which makes the efficient conjugate gradient (CG) method directly applicable. Appendix~\ref{app:matrix_free_operators} derives the required operator expressions and details the sensitivity computation procedure.

\vspace{-3mm}
\section{Experiments}
\vspace{-2mm}
We evaluate \solver{lapanda} through three groups of experiments on
representative nonconvex problems with general constraints, focusing on
computational efficiency and sensitivity accuracy, learning performance, and
embeddable implementation. For each experimental setting, we compare with
representative solvers from Table~\ref{tab:related_solver_comparison} that
support the corresponding problem structure.
\vspace{-2mm}
\subsection{Nonconvex Constrained Rosenbrock Benchmark}
\vspace{-2mm}
We first consider a nonconvex Rosenbrock-chain benchmark with smooth nonlinear constraints:
\begin{equation}
	\begin{aligned}
		\min_{x\in\mathbb{R}^n}\quad &
		\sum_{i=0}^{n-2}
		\left[
		\theta_0(x_{i+1}-x_i^2)^2
		+
		\theta_1(1-x_i)^2
		\right]
		+
		\frac{1}{2}\theta_2\|x\|_2^2 \\
		\mathrm{s.t.}\quad &
		-2 \le x_i \le 2,\qquad i=0,\ldots,n-1,\\
		&
		g_i(x,\theta)
		=
		x_i^2+x_{i+1}^2-r_i(\theta)^2
		\le 0,\qquad i=0,\ldots,n-2,
	\end{aligned}
\end{equation}
where
\(
r_i(\theta)
=
\theta_3
+
0.1\sin\left(\frac{2\pi(i+1)}{n}\right).
\)
We vary the problem dimension \(n\) and compare \solver{lapanda} with two
baselines. \solver{CasADi-IPOPT} is a widely used baseline for general
nonlinear programs, while \solver{Explicit KKT} explicitly forms the original problem's
KKT sensitivity system in Eq.~(\ref{eq:original_kkt_sensitivity}) and solves
this indefinite saddle-point system using MINRES. With the relative
gradient error kept below \(2\%\), we compare their
computation times and memory consumption.

\begin{table}[h]
	\centering
	\caption{Matched-accuracy comparison on the constrained Rosenbrock benchmark.}
	\label{tab:lapanda_casadi_benchmark}
	\setlength{\tabcolsep}{4pt}
	\resizebox{0.92\linewidth}{!}{
		\begin{tabular}{c|ccc|ccc|ccc}
			\toprule
			\multirow{2}{*}{\(n\)}
			& \multicolumn{3}{c|}{Forward time (ms)}
			& \multicolumn{3}{c|}{Backward time (ms)}
			& \multicolumn{3}{c}{Memory (MB)} \\
			\cmidrule(lr){2-4}
			\cmidrule(lr){5-7}
			\cmidrule(lr){8-10}
			& \cellcolor{lapandahighlight}lapanda & Explicit & CasADi
			& \cellcolor{lapandahighlight}lapanda & Explicit & CasADi
			& \cellcolor{lapandahighlight}lapanda & Explicit & CasADi \\
			\midrule
			100  & \cellcolor{lapandahighlight}\textbf{8.23}   & -- & 36.50   & \cellcolor{lapandahighlight}4.04            & 4.22            & \textbf{3.04} & \cellcolor{lapandahighlight}\textbf{1.78} & 5.01  & 17.55 \\
			200  & \cellcolor{lapandahighlight}\textbf{14.37}  & -- & 108.19  & \cellcolor{lapandahighlight}\textbf{5.89}  & 10.23           & 17.05          & \cellcolor{lapandahighlight}\textbf{2.06} & 7.78  & 29.39 \\
			500  & \cellcolor{lapandahighlight}\textbf{49.75}  & -- & 496.03  & \cellcolor{lapandahighlight}\textbf{14.36} & 51.26           & 245.99         & \cellcolor{lapandahighlight}\textbf{3.12} & 25.32 & 84.70 \\
			1000 & \cellcolor{lapandahighlight}\textbf{200.00} & -- & 1968.78 & \cellcolor{lapandahighlight}\textbf{37.36} & 218.21          & 2325.16        & \cellcolor{lapandahighlight}\textbf{5.31} & 79.65 & 249.66 \\
			\bottomrule
		\end{tabular}
			}
\end{table}

\begin{wrapfigure}{r}{0.4\textwidth}
	\vspace{-4mm}
	\includegraphics[width=\linewidth]{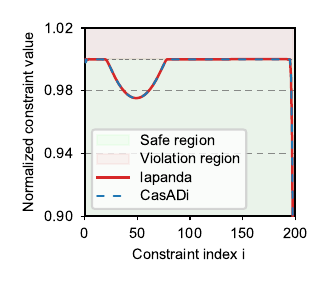}
	\vspace{-10mm}
	\caption{
		Normalized constraint values 
		$\sqrt{x_i^2+x_{i+1}^2}/r_i(\theta)$ for $n=200$.
	}
	\label{fig:constraint_values_N200}
	\vspace{-2em}
\end{wrapfigure}

The main results are reported in Table~\ref{tab:lapanda_casadi_benchmark}.
\solver{lapanda} is generally faster than both baselines, particularly for
larger problem instances. The forward speedup mainly stems from the efficient
\solver{PANDA} inner solver, while the backward advantage comes from the
lightweight operator-based implementation. Its matrix-free formulation also
provides a clear memory advantage. In the representative case shown in
Fig.~\ref{fig:constraint_values_N200}, most nonlinear constraints are active and
are handled well by the augmented Lagrangian method of \solver{lapanda}.
Appendix~\ref{app:rosenbrock_experiment_details} details the experimental
settings, analyzes the sensitivity-error behavior predicted by
Theorem~\ref{thm:sensitivity_alignment}, and describes both the strategies used
to control this error and the accuracy-matching protocol used in this benchmark
to ensure a fair comparison.

\subsection{Imitation Learning with Constrained OCPs}
\vspace{-2mm}
We next evaluate \solver{lapanda} on imitation learning tasks built from constrained
OCPs. The goal is to learn problem parameters such that the
MPC solution matches the target trajectory generated by a teacher controller.
Accordingly, the imitation loss measures the discrepancy between the planned
state-control trajectory and the demonstrated trajectory:
\vspace{-2mm}
\begin{equation}
	\mathcal{L}_{\mathrm{imit}}(\theta)
	=
	\frac{1}{2N_{\mathrm{demo}}}
	\sum_{j=1}^{N_{\mathrm{demo}}}
	\left(
	\left\|
	X^\star(\theta;\xi_j)
	-
	X_j^{\mathrm{demo}}
	\right\|_2^2
	+
	\left\|
	U^\star(\theta;\xi_j)
	-
	U_j^{\mathrm{demo}}
	\right\|_2^2
	\right),
	\label{eq:imit_loss}
\end{equation}
where \(N_{\mathrm{demo}}\) is the number of demonstrations, and
\(X^\star(\theta;\xi_j),U^\star(\theta;\xi_j)\) denote the solution
trajectory under parameter \(\theta\) and initial condition \(\xi_j\).
We consider three representative constrained OCPs:
\vspace{-4mm}
\begin{itemize}
	\item \textbf{CartPole}: nonlinear underactuated dynamics with force bounds
	and a total control-energy constraint.
	\item \textbf{Quadrotor}: nonlinear flight dynamics with thrust and
	torque bounds, altitude and horizontal-position bounds, and attitude
	constraints.
	\item \textbf{Robot Arm}: linear joint-space dynamics with joint-velocity
	bounds and joint-position limits, together with nonlinear, nonconvex
	end-effector obstacle-avoidance geometry.
\end{itemize}
\vspace{-2mm}
\paragraph{Open-loop learning.}
We first conduct open-loop imitation learning experiments. In this setting,
each sample starts from an independent initial condition, and learning is driven by the discrepancy between the OCP solutions obtained with the current parameters and the teacher solutions generated using the target parameters.
This setting evaluates whether the computed sensitivities support stable and efficient parameter learning across different initial conditions.

\begin{figure}[h]
	\centering
	\includegraphics[
	width=\linewidth,
	trim={0 0 0 2mm},
	clip
	]{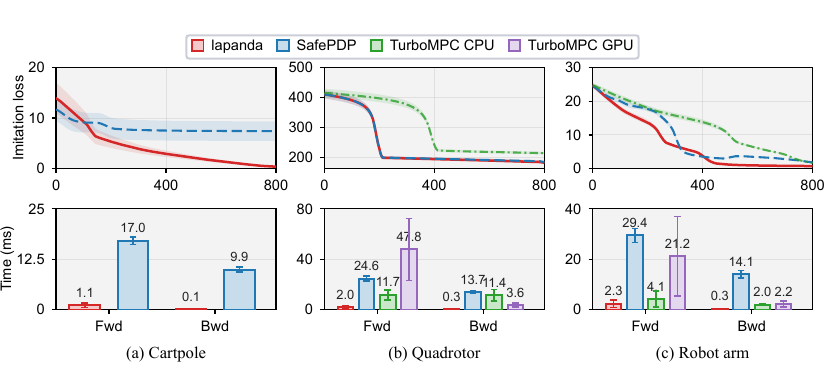}
	\vspace{-2em}
	\caption{
		Open-loop imitation learning results on different constrained OCPs.
	}
	\vspace{-4mm}
	\label{fig:ocp_learning_results}
\end{figure}

As shown in Fig.~\ref{fig:ocp_learning_results}, we summarize the loss trends
and average solve times on different constrained OCPs, comparing \solver{lapanda} with
\solver{SafePDP}, an efficient PMP-based differentiable OCP solver, and \solver{TurboMPC}, a
recent GPU-parallel differentiable MPC solver. \solver{TurboMPC} fails to converge on
the CartPole task because it mainly supports
pointwise inequalities, whereas this task includes a trajectory-level
control-energy constraint. Across the tested tasks, \solver{lapanda} achieves stable
learning and the fastest solve time for these OCPs with multiple
constraints. In comparison, the GPU implementation of \solver{TurboMPC} is less
effective at this problem scale, consistent with its reported constant-factor
overhead from sparse direct linear solves.
\paragraph{Closed-loop learning.}
We further evaluate closed-loop imitation learning. Given a demonstration trajectory generated by the target parameters, the learner repeatedly rolls out the system using the current parameters in a receding-horizon manner. Learning is driven by minimizing the discrepancy between the resulting closed-loop trajectory and the demonstration, thereby directly evaluating whether the learned parameters recover the desired closed-loop policy.

\begin{figure}[h]
	\centering
	\includegraphics[
	width=\linewidth,
	trim={0 0 0 1mm},
	clip
	]{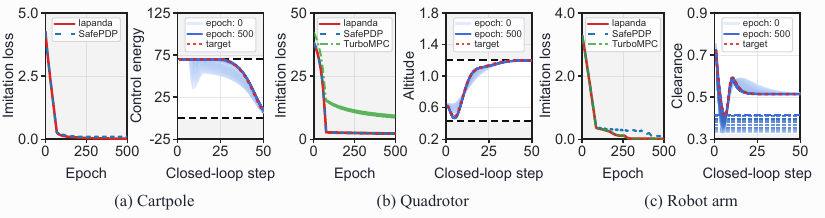}
	\vspace{-2em}
	\caption{
		Closed-loop imitation learning results including training
		loss and policy evolution.
	}
	\label{fig:closed_loop_learning_results}
\end{figure}

As shown in Fig.~\ref{fig:closed_loop_learning_results}, \solver{lapanda} achieves stable policy learning in the closed-loop setting. The constraint evolution shows that \solver{lapanda} maintains the imposed constraints within the prescribed tolerance throughout training. In the Robot Arm task, the learned behavior also involves learning the constraint-related parameter, as reflected by the change in the obstacle-avoidance constraint during training. More detailed computational results are reported in Table~\ref{tab:closed_loop_ocp_results}.
\vspace{-4mm}
\begin{table}[h]
	\centering
	\caption{Detailed performance comparison of different solvers on closed-loop OCP benchmarks.}
	\label{tab:closed_loop_ocp_results}
	\resizebox{0.9\columnwidth}{!}{%
		\begin{tabular}{clcccc}
			\toprule
			OCP & Method & Fwd. time (ms) & Bwd. time (ms) & Total time (ms) & Memory (MB) \\
			\midrule
			
			\multirow{2}{*}{CartPole}
			& \cellcolor{lapandahighlight}lapanda
			& \cellcolor{lapandahighlight}\textbf{5.10}
			& \cellcolor{lapandahighlight}\textbf{0.09}
			& \cellcolor{lapandahighlight}\textbf{5.20}
			& \cellcolor{lapandahighlight}\textbf{1.64} \\
			
			& SafePDP
			& 22.06
			& 9.65
			& 31.71
			& 7.71 \\
			
			\midrule
			
			\multirow{4}{*}{Quadrotor}
			& \cellcolor{lapandahighlight}lapanda
			& \cellcolor{lapandahighlight}\textbf{1.20}
			& \cellcolor{lapandahighlight}\textbf{0.24}
			& \cellcolor{lapandahighlight}\textbf{1.45}
			& \cellcolor{lapandahighlight}\textbf{1.79} \\
			
			& SafePDP
			& 29.86
			& 13.41
			& 43.27
			& 9.23 \\
			
			& TurboMPC-CPU
			& 10.46
			& 7.74
			& 18.20
			& 241.30 \\
			
			& TurboMPC-GPU
			& 21.78
			& 2.28
			& 24.06
			& 336.48 \\
			
			\midrule
			
			\multirow{4}{*}{Robot Arm}
			& \cellcolor{lapandahighlight}lapanda
			& \cellcolor{lapandahighlight}\textbf{2.36}
			& \cellcolor{lapandahighlight}\textbf{0.29}
			& \cellcolor{lapandahighlight}\textbf{2.66}
			& \cellcolor{lapandahighlight}\textbf{1.76} \\
			
			& SafePDP
			& 27.94
			& 12.85
			& 40.79
			& 8.02 \\
			
			& TurboMPC-CPU
			& 5.62
			& 1.17
			& 6.79
			& 248.35 \\
			
			& TurboMPC-GPU
			& 24.09
			& 1.96
			& 26.05
			& 347.11 \\
			
			\bottomrule
		\end{tabular}
	}
\end{table}

\vspace{-4mm}
The solver runtimes exhibit a trend similar to that observed in the open-loop experiments. For memory usage, we report the sum of the solver construction overhead and
the peak temporary memory overhead during the solve. \solver{TurboMPC} incurs a heavier memory footprint due to its GPU and parallel implementation, whereas \solver{lapanda} remains substantially more lightweight. Detailed problem formulations and additional experimental results are provided in Appendix~\ref{app:ocp_experiment_details}, and all solvers in the above experiments use the same stopping tolerance of $10^{-3}$.
\subsection{Embedded Deployment}
\begin{wrapfigure}{r}{0.48\linewidth}
	\vspace{-6mm}
	\centering
	\includegraphics[width=0.92\linewidth,trim={0mm 0mm 0mm 0mm},clip]{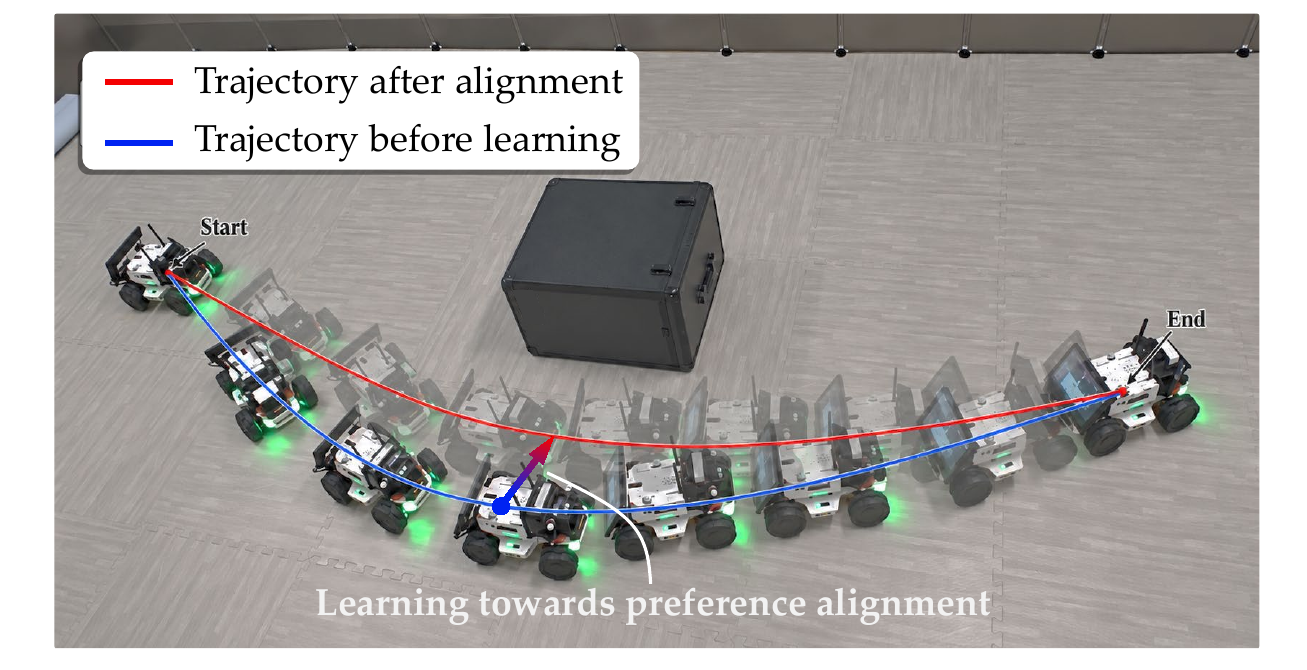}
	\vspace{-4mm}
	\caption{Mobile robot trajectory alignment.}
	\label{fig:embedded_platform}
	\vspace{-6mm}
\end{wrapfigure}

Finally, we evaluate \solver{lapanda} on an embedded mobile robot equipped with a six-core Arm\textsuperscript{\textregistered} Cortex\textsuperscript{\textregistered}-A78AE processor. We consider two obstacle-avoidance tasks: circular obstacle avoidance with smooth nonlinear constraints and rectangular obstacle avoidance formulated as an MPCC. We compare \solver{lapanda} with \solver{acados}, a representative embedded nonlinear optimal control solver, and summarize the results in Table~\ref{tab:embedded_results}.
\vspace{-8mm}
\begin{table}[h]
	\centering
	\caption{Embedded deployment performance comparison.}
	\label{tab:embedded_results}
	\begin{tabular}{clcccc}
		\toprule
		Obstacle & Method & Success & Fwd. (ms) & Bwd. (ms) & Memory (MB) \\
		\midrule
		
		\multirow{2}{*}{Circle}
		& \cellcolor{lapandahighlight}lapanda
		& \cellcolor{lapandahighlight}\cmark
		& \cellcolor{lapandahighlight}1.76
		& \cellcolor{lapandahighlight}0.16
		& \cellcolor{lapandahighlight}2.16 \\
		
		& acados
		& \cmark
		& 3.27
		& 0.24
		& 4.52 \\
		
		\midrule
		
		\multirow{2}{*}{Rectangle}
		& \cellcolor{lapandahighlight}lapanda
		& \cellcolor{lapandahighlight}\cmark
		& \cellcolor{lapandahighlight}55.5
		& \cellcolor{lapandahighlight}1.13
		& \cellcolor{lapandahighlight}2.60 \\
		
		& acados
		& \xmark
		& --
		& --
		& -- \\
		
		\bottomrule
	\end{tabular}
\end{table}

\vspace{-4mm}
As summarized in Table~\ref{tab:embedded_results}, both \solver{lapanda} and \solver{acados} successfully solve the circular obstacle-avoidance task. 
The rectangular task is more challenging because its MPCC formulation violates
LICQ and thereby induces a singular KKT matrix. Consequently, \solver{acados},
which relies on local KKT models, fails to converge under the tested
configurations.
In contrast, the first-order \solver{PANDA} inner solver of \solver{lapanda} relies on a proximal-residual optimality measure rather than KKT systems. Combined with ALM constraint penalization, it can progressively approach the feasible set from infeasible iterates, empirically converging outside our theoretical assumptions and supporting successful imitation learning. The learning process is shown in Fig.~\ref{fig:embedded_learning}. We further evaluate both solvers using a smooth approximation of the rectangular-obstacle constraint in Appendix~\ref{app:embedded_experiment_details}, where the detailed problem formulations and solver configurations are provided.
\vspace{-2mm}
\begin{figure}[h]
	\centering
	\includegraphics[
	width=\linewidth,
	trim={0 0mm 0 2mm},
	clip
	]{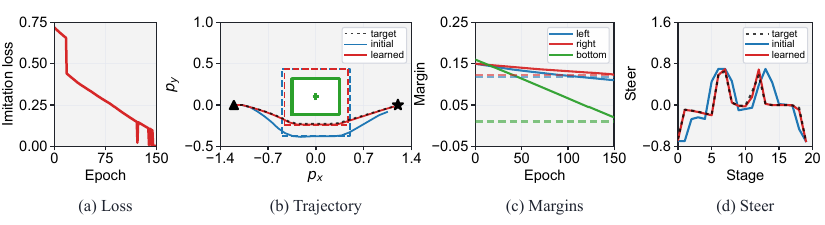}
	\vspace{-7mm}
	\caption{Learning performance for the rectangular obstacle-avoidance imitation task.}
	\label{fig:embedded_learning}
\end{figure}
\vspace{-6mm}
\section{Conclusion}
\vspace{-3mm}
We proposed \solver{lapanda}, a matrix-free differentiable solver for nonconvex optimization layers with general constraints. We established its optimization and sensitivity principles and validated its effectiveness through extensive experiments, demonstrating fast computation, low memory consumption, reliable constraint satisfaction, and stable learning performance. With efficient implementations and user-friendly interfaces across multiple platforms (Appendix~D), \solver{lapanda} provides a practical foundation for extending differentiable optimization to broader learning and control applications.

\textbf{Limitations and future work.}
First, the finite ALM penalties used in \solver{lapanda} introduce approximation errors in the backward sensitivity. Although stable learning is observed across most experiments, further theoretical characterization and matrix-free differentiation of the original KKT system warrant investigation. Second, \solver{lapanda} requires careful tuning of multiplier initialization, penalty updates, and \solver{PANDA} parameters. Adaptive parameter strategies could therefore reduce manual tuning and improve robustness. Third, as a differentiable solver for nonconvex optimization layers with general constraints, \solver{lapanda} does not yet exploit the temporal sparsity and dynamical structure specific to OCPs, leaving room for further scalability improvements. Finally, the current implementation targets lightweight CPU execution. Extending \solver{lapanda} to GPU-parallel and batched computation represents a natural next step.

\subsection*{Acknowledgments}
This research was supported in part by the Japan Science and Technology Agency
(JST), the Establishment of University Fellowships Toward the Creation of
Science Technology Innovation, under Grant JPMJFS2132, and in part under Grant
JPMJSP2136.

\bibliography{lapanda_arxiv}
\bibliographystyle{iclr2027_conference}

\clearpage
\appendix

\section*{\centering Appendix}
\vspace{-0.5em}

\section{Problem Definition and OCP Transcription}
\label{app:ocp_transcription}

This appendix presents the single-shooting transcription adopted for constrained optimal control problems. Consider the finite-horizon problem
\begin{equation}
	\begin{aligned}
		\min_{\{x_k,u_k\}}
		\quad &
		\sum_{k=0}^{N-1}
		\ell_k(x_k,u_k;\theta)
		+
		\ell_N(x_N;\theta)
		\\
		\mathrm{s.t.}\quad
		&
		x_0=x_{\mathrm{init}}(\theta),
		\\
		&
		x_{k+1}=f_k(x_k,u_k;\theta),
		\\
		&
		c_k(x_k,u_k;\theta)=0,
		\quad
		h_k(x_k,u_k;\theta)\le 0,
		\\
		&
		c_N(x_N;\theta)=0,
		\quad
		h_N(x_N;\theta)\le 0,
		\\
		&
		u_k\in\mathbf{U}_k,
		\quad k=0,\dots,N-1 .
	\end{aligned}
	\label{eq:app_ocp}
\end{equation}

In single shooting, the state variables are eliminated by recursively applying the system dynamics. For the stacked control vector
$
u:=(u_0^\top,\ldots,u_{N-1}^\top)^\top,
$
the state rollout is defined as
\begin{equation}
	x_0(u;\theta)=x_{\mathrm{init}}(\theta),\qquad
	x_{k+1}(u;\theta)=f_k(x_k(u;\theta),u_k;\theta).
	\label{eq:app_single_shooting_rollout}
\end{equation}
Substituting Eq.~(\ref{eq:app_single_shooting_rollout}) into the objective and constraints gives
\begin{subequations}
\begin{equation}
	\ell(u,\theta)
	:=
	\sum_{k=0}^{N-1}\ell_k(x_k(u;\theta),u_k;\theta)
	+
	\ell_N(x_N(u;\theta);\theta),
	\label{eq:app_ocp_objective}
\end{equation}
and
\begin{equation}
	c(u,\theta):=
	\begin{bmatrix}
		c_0(x_0(u;\theta),u_0;\theta)\\
		\vdots\\
		c_{N-1}(x_{N-1}(u;\theta),u_{N-1};\theta)\\
		c_{N}(x_{N}(u;\theta);\theta)
	\end{bmatrix},
	\quad
	h(u,\theta):=
	\begin{bmatrix}
		h_0(x_0(u;\theta),u_0;\theta)\\
		\vdots\\
		h_{N-1}(x_{N-1}(u;\theta),u_{N-1};\theta)\\
		h_{N}(x_{N}(u;\theta);\theta)
	\end{bmatrix}.
	\label{eq:app_ocp_constraints}
\end{equation}
\end{subequations}

The simple constraint set becomes
$
\mathbf U:=\mathbf U_0\times\cdots\times\mathbf U_{N-1}.
$
Therefore, Eq.~(\ref{eq:app_ocp}) is reformulated as the control-only constrained optimization problem
\begin{equation}
	\min_{u\in\mathbf U}\ \ell(u,\theta)
	\quad\mathrm{s.t.}\quad
	c(u,\theta)=0,\qquad
	h(u,\theta)\le0,
\end{equation}
which is the form of Problem~(\ref{eq:problem_p}).

\section{Theoretical Details}

\subsection{Eliminating the Slack Variable}
\label{app:alm_reduction}

For fixed \(u\), the \(z\)-subproblem in Eq.~(\ref{eq:alm_lagrangian}) is
\begin{equation}
	\min_{z\in\mathcal D}
	\lambda^\top(G(u,\theta)-z)
	+
	\frac{\rho}{2}\|G(u,\theta)-z\|^2 .
\end{equation}

Completing the square gives
\begin{equation}
	\frac{\rho}{2}
	\left\|
	G(u,\theta)+\rho^{-1}\lambda-z
	\right\|^2
	-
	\frac{1}{2\rho}\|\lambda\|^2,
\end{equation}
so that
\begin{equation}
	z^\star(u)
	=
	\Pi_{\mathcal D}\left(G(u,\theta)+\rho^{-1}\lambda\right).
\end{equation}

Substituting this minimizer into the augmented Lagrangian and dropping the
constant term \(-\|\lambda\|^2/(2\rho)\) yields the reduced objective
\begin{equation*}
	\psi_\rho(u;\theta,\lambda)
	=
	\ell(u,\theta)
	+
	\frac{\rho}{2}
	\operatorname{dist}^2
	\left(G(u,\theta)+\rho^{-1}\lambda,\mathcal D\right),
\end{equation*}

which gives Eq.~(\ref{eq:reduced_alm_subproblem}) and
Eq.~(\ref{eq:reduced_alm_objective}).

\subsection{PANDA Inner Solver}
\label{app:panda_composite_subproblem}

Given the ALM variables \((\lambda_r,\rho_r)\), the reduced inner
subproblem takes the composite form
\begin{equation}
	\min_u\ 
	\varphi_r(u)
	:=
	\psi_{\rho_r}(u;\theta,\lambda_r)
	+
	\delta_{\mathbf U}(u),
	\label{eq:panda_inner_composite}
\end{equation}
where \(\psi_{\rho_r}\) is the smooth reduced ALM objective and
\(\delta_{\mathbf U}\) encodes the simple constraint on \(u\).
For a stepsize \(\gamma>0\), define the forward-backward mapping
\begin{equation}
	T_{\gamma,r}(u)
	:=
	\operatorname{prox}_{\gamma\delta_{\mathbf U}}
	\left(
	u-\gamma\nabla_u\psi_{\rho_r}(u;\theta,\lambda_r)
	\right)
	=
	\Pi_{\mathbf U}
	\left(
	u-\gamma\nabla_u\psi_{\rho_r}(u;\theta,\lambda_r)
	\right),
	\label{eq:panda_fbm}
\end{equation}
and the associated forward-backward residual
\begin{equation}
	R_{\gamma,r}(u)
	:=
	\frac{1}{\gamma}
	\left(
	u-T_{\gamma,r}(u)
	\right).
	\label{eq:panda_residual}
\end{equation}
A point satisfying \(R_{\gamma,r}(u)=0\) is an \(R\)-critical point.
For the indicator term \(\delta_{\mathbf U}\), this condition is equivalent to
\begin{equation}
	0\in
	\nabla_u\psi_{\rho_r}(u;\theta,\lambda_r)
	+
	N_{\mathbf U}(u),
	\label{eq:panda_stationarity}
\end{equation}
and hence it characterizes a stationary point of the reduced ALM subproblem.

\solver{PANDA} solves the residual equation \(R_{\gamma,r}(u)=0\) by combining three
mechanisms. (i) It uses the forward-backward mapping as a descent step for the
composite objective and uses the residual norm as the stationarity measure.
(ii) It constructs a limited-memory quasi-Newton correction for the residual
equation, so that the local convergence is faster than a plain projected
gradient iteration. (iii) It globalizes the quasi-Newton correction through an
averaged trial point and an FBE-based line search. In addition, \solver{PANDA} adapts the
stepsize \(\gamma\): it reduces \(\gamma\) when the local upper model is not
valid, and attempts to enlarge \(\gamma\) when the current stepsize becomes too
conservative.

The simplified inner routine is summarized in
Algorithm~\ref{alg:panda_inner}. The output is an approximate \(R\)-critical
point of Eq.~(\ref{eq:panda_inner_composite}), which is later used for implicit
differentiation through \(R_{\gamma,r}(u_{r+1})=0\).

\begin{algorithm}[h]
	\caption{PANDA inner solver for the reduced ALM subproblem}
	\label{alg:panda_inner}
	\begin{algorithmic}[1]
		\STATE \textbf{Input:} \(u^0\), \(\gamma_0>0\), tolerance
		\(\varepsilon_{\mathrm{in}}\), ALM variables \((\lambda_r,\rho_r)\)
		\FOR{\(k=0,1,\ldots\)}
		\STATE Compute \(\bar u^k=T_{\gamma_k,r}(u^k)\) and
		\(R_{\gamma_k,r}(u^k)=\gamma_k^{-1}(u^k-\bar u^k)\).
		\IF{\(\|R_{\gamma_k,r}(u^k)\|\le \varepsilon_{\mathrm{in}}\)}
		\STATE \textbf{return} \(u^{k}\)
		\ENDIF
		\STATE Construct a quasi-Newton direction
		\(d^k\approx -H_kR_{\gamma_k,r}(u^k)\).
		\STATE Form the averaged trial point
		\(u_{\mathrm{trial}}^k=(1-\sigma_k)\bar u^k+\sigma_k(u^k+d^k)\),
		with \(\sigma_k\in(0,1]\).
		\STATE Backtrack \(\sigma_k\) until the FBE sufficient decrease condition holds.
		\STATE Adapt \(\gamma_k\) according to the local upper-model condition.
		\STATE Set \(u^{k+1}=u_{\mathrm{trial}}^k\).
		\ENDFOR
	\end{algorithmic}
\end{algorithm}

\subsection{Proof of Theorem~\ref{thm:lapanda_wellposed}}
\label{app:convergence}

We prove the three statements in Theorem~\ref{thm:lapanda_wellposed}. The
argument follows the local method-of-multipliers proof in
\citet[Proposition~4.2.3]{bertsekas1999nonlinear} and the
positive-definiteness argument in \citet[Theorem~17.5]{nocedal2006numerical},
with the difference that we apply the argument to the slack reformulation and
retain the simple constraints \(\mathbf U\times\mathcal D\).

Set \(y=(u,z)\), \(\mathcal C=\mathbf U\times\mathcal D\), and
\(r(y,\theta)=G(u,\theta)-z\). Then the slack problem is locally
\begin{equation}
	\min_{y\in\mathcal C}\ \ell(u,\theta)
	\quad\mathrm{s.t.}\quad
	r(y,\theta)=0.
	\label{eq:app_slack_problem_local}
\end{equation}
By strict complementarity, the active constraints of \(\mathcal C\) are locally
identified. We denote the corresponding active manifold by \(b(y)=0\), with
multiplier \(\nu\), and define
\begin{equation}
	\mathscr L(y,\lambda,\nu;\theta)
	=
	\ell(u,\theta)
	+
	\langle\lambda,r(y,\theta)\rangle
	+
	\langle\nu,b(y)\rangle.
	\label{eq:app_slack_lagrangian}
\end{equation}
Let \(y^\star=(u^\star,z^\star)\),
\(J_r^\star=\nabla_y r(y^\star,\theta)\),
\(J_b^\star=\nabla b(y^\star)\), and
\(W^\star=\nabla_{yy}^2\mathscr L(y^\star,\lambda^\star,\nu^\star;\theta)\).
LICQ means that
\(\begin{bmatrix}J_r^\star\\ J_b^\star\end{bmatrix}\)
has full row rank, and SOSC means that
\(d^\top W^\star d>0\) for all nonzero \(d\) satisfying
\(J_r^\star d=0\) and \(J_b^\star d=0\).

\paragraph{Claim (i).}
For fixed \((\lambda_r,\rho_r)\), the slack ALM subproblem is
\begin{equation}
	\min_{y\in\mathcal C}
	\quad
	\ell(u,\theta)
	+
	\lambda_r^\top r(y,\theta)
	+
	\frac{\rho_r}{2}\|r(y,\theta)\|^2,
	\label{eq:slack_alm_subproblem_proof}
\end{equation}
The local KKT conditions of
Eq.~(\ref{eq:slack_alm_subproblem_proof}) are
\begin{equation}
	\begin{aligned}
		0
		&=
		\nabla_y\ell(u,\theta)
		+
		\nabla_y r(y,\theta)^\top
		\left(
		\lambda_r+\rho_r r(y,\theta)
		\right)
		+
		\nabla b(y)^\top\nu,\\
		0
		&=
		b(y).
	\end{aligned}
	\label{eq:slack_alm_subproblem_kkt}
\end{equation}
Define the shifted multiplier
\(\lambda:=\lambda_r+\rho_r r(y,\theta)\). Then
Eq.~(\ref{eq:slack_alm_subproblem_kkt}), together with the multiplier relation,
can be written as
\begin{equation}
	\begin{aligned}
		0
		&=
		\nabla_y\ell(u,\theta)
		+
		\nabla_y r(y,\theta)^\top\lambda
		+
		\nabla b(y)^\top\nu,\\
		0
		&=
		r(y,\theta)
		-
		\rho_r^{-1}
		\left(
		\lambda-\lambda_r
		\right),\\
		0
		&=
		b(y).
	\end{aligned}
	\label{eq:slack_alm_kkt_with_lambda}
\end{equation}

We first verify the second-order condition for the ALM subproblem. At
\(y^\star\), since \(r(y^\star,\theta)=0\), the Hessian of the ALM objective on
the identified active manifold is
\(W^\star+\rho J_r^{\star\top}J_r^\star\), where
\(W^\star=\nabla_{yy}^2
\mathscr L(y^\star,\lambda^\star,\nu^\star;\theta)\).
By LICQ and SOSC, there exists \(\bar\rho>0\) such that
\begin{equation}
	d^\top
	\left(
	W^\star+\rho J_r^{\star\top}J_r^\star
	\right)d
	>0
	\quad
	\forall d\neq0,\quad
	J_b^\star d=0,\quad
	\rho\ge\bar\rho.
	\label{eq:alm_pd_reference}
\end{equation}
Indeed, if this were false, then there would exist
\(\rho_k\to\infty\) and \(\|d_k\|=1\), with \(J_b^\star d_k=0\), such that
\begin{equation}
	d_k^\top W^\star d_k
	+
	\rho_k\|J_r^\star d_k\|^2
	\le0.
	\label{eq:app_pd_contradiction_sequence}
\end{equation}
Hence \(J_r^\star d_k\to0\). Taking a convergent subsequence gives a nonzero
\(\bar d\) satisfying \(J_b^\star\bar d=0\) and
\(J_r^\star\bar d=0\), which contradicts SOSC.

Set \(t:=\rho_r^{-1}(\lambda_r-\lambda^\star)\) and
\(\kappa:=\rho_r^{-1}\). Then
Eq.~(\ref{eq:slack_alm_kkt_with_lambda}) is equivalently written as
\begin{equation}
	\Phi(y,\lambda,\nu;t,\kappa)=0,
	\label{eq:alm_ift_system}
\end{equation}
where
\[
\Phi(y,\lambda,\nu;t,\kappa)
=
\begin{bmatrix}
	\nabla_y\ell(u,\theta)
	+
	\nabla_y r(y,\theta)^\top\lambda
	+
	\nabla b(y)^\top\nu\\[1mm]
	r(y,\theta)
	+
	t
	+
	\kappa\lambda^\star
	-
	\kappa\lambda\\[1mm]
	b(y)
\end{bmatrix}.
\label{eq:alm_phi_def}
\]
For \(t=0\), the point
\((y^\star,\lambda^\star,\nu^\star)\) solves
Eq.~(\ref{eq:alm_ift_system}) for all
\(\kappa\in[0,\bar\rho^{-1}]\). The Jacobian of \(\Phi\) with respect to
\((y,\lambda,\nu)\) at this point is
\begin{equation}
	K_\kappa
	=
	\begin{bmatrix}
		W^\star & J_r^{\star\top} & J_b^{\star\top}\\
		J_r^\star & -\kappa I & 0\\
		J_b^\star & 0 & 0
	\end{bmatrix}.
	\label{eq:alm_ift_jacobian}
\end{equation}

We claim that \(K_\kappa\) is nonsingular for all
\(\kappa\in[0,\bar\rho^{-1}]\). For \(\kappa=0\), this follows from the
standard KKT matrix nonsingularity under LICQ and SOSC. For \(\kappa>0\),
suppose \(K_\kappa(d,p,s)=0\). Then
\begin{equation}
	\begin{aligned}
		W^\star d
		+
		J_r^{\star\top}p
		+
		J_b^{\star\top}s
		&=0,\\
		J_r^\star d-\kappa p
		&=0,\\
		J_b^\star d
		&=0.
	\end{aligned}
	\label{eq:kappa_null_system}
\end{equation}
Multiplying the first equation in
Eq.~(\ref{eq:kappa_null_system}) by \(d^\top\) and using the second and third
equations gives
\begin{equation}
	d^\top W^\star d
	+
	\kappa^{-1}\|J_r^\star d\|^2
	=0.
	\label{eq:app_kappa_energy_identity}
\end{equation}
By Eq.~(\ref{eq:alm_pd_reference}), this implies \(d=0\). Then
\(J_r^{\star\top}p+J_b^{\star\top}s=0\), and LICQ gives \(p=0\) and \(s=0\).
Thus \(K_\kappa\) is nonsingular.

Since \(K_\kappa\) depends continuously on \(\kappa\) and is nonsingular on
the compact interval \([0,\bar\rho^{-1}]\), its inverse is uniformly bounded.
The parameterized implicit function theorem therefore provides a common local
neighborhood for all \(\kappa\in[0,\bar\rho^{-1}]\). Consequently, there exist
\(\delta>0\) and a neighborhood of
\((y^\star,\lambda^\star,\nu^\star)\) such that, whenever
\(\rho_r\ge\bar\rho\) and
\(\|\lambda_r-\lambda^\star\|\le\rho_r\delta\), the system in
Eq.~(\ref{eq:alm_ift_system}) has a unique local solution
\((y_{r+1},\lambda_{r+1},\nu_{r+1})\). Equivalently,
\(y_{r+1}=(u_{r+1},z_{r+1})\) is the unique local solution of the slack ALM
subproblem in Eq.~(\ref{eq:slack_alm_subproblem_proof}), and
\(\lambda_{r+1}
=
\lambda_r+\rho_r r(y_{r+1},\theta)\).

Finally, by taking the neighborhood smaller if necessary, strict
complementarity keeps the active set of \(\mathcal C\) fixed, LICQ persists by
continuity of the active constraint gradients, and the positive-definiteness
condition in Eq.~(\ref{eq:alm_pd_reference}) gives SOSC at \(y_{r+1}\). This
proves claim (i).

\paragraph{Claim (ii).}
From the implicit function construction in claim (i), there exist local \(C^1\) functions \(y(t,\kappa)\) and \(\lambda(t,\kappa)\) such that
\(y_{r+1}=y(t_r,\kappa_r)\) and
\(\lambda_{r+1}=\lambda(t_r,\kappa_r)\), where
\(t_r=\rho_r^{-1}(\lambda_r-\lambda^\star)\) and
\(\kappa_r=\rho_r^{-1}\). Moreover,
\(y(0,\kappa)=y^\star\) and
\(\lambda(0,\kappa)=\lambda^\star\) for all
\(\kappa\in[0,\bar\rho^{-1}]\).

Since the implicit functions are \(C^1\), and since
\(\kappa\in[0,\bar\rho^{-1}]\) belongs to a compact interval, their derivatives
with respect to \(t\) are uniformly bounded. Hence, there exists \(M>0\) such
that, for all sufficiently small \(t\) and all
\(\kappa\in[0,\bar\rho^{-1}]\),
\begin{equation}
	\|y(t,\kappa)-y(0,\kappa)\|
	+
	\|\lambda(t,\kappa)-\lambda(0,\kappa)\|
	\le
	M\|t\|.
	\label{eq:alm_ift_lipschitz}
\end{equation}
Taking \(t=t_r\) and \(\kappa=\kappa_r\) in
Eq.~(\ref{eq:alm_ift_lipschitz}) gives
\begin{equation}
	\|y_{r+1}-y^\star\|
	+
	\|\lambda_{r+1}-\lambda^\star\|
	\le
	M\|t_r\|.
	\label{eq:alm_local_estimate_pre}
\end{equation}
Using \(t_r=\rho_r^{-1}(\lambda_r-\lambda^\star)\), we obtain
\begin{equation}
	\|y_{r+1}-y^\star\|
	+
	\|\lambda_{r+1}-\lambda^\star\|
	\le
	M
	\frac{\|\lambda_r-\lambda^\star\|}{\rho_r}.
	\label{eq:alm_local_estimate}
\end{equation}
Since \(y=(u,z)\), Eq.~(\ref{eq:alm_local_estimate}) directly implies the estimates
in claim (ii).

\paragraph{Claim (iii).}
By Appendix~\ref{app:alm_reduction}, eliminating \(z\) gives
\begin{equation*}
	z_{\rho_r}(u)
	=
	\Pi_{\mathcal D}
	\left(
	G(u,\theta)+\rho_r^{-1}\lambda_r
	\right),
\end{equation*}
and the reduced composite subproblem
\begin{equation*}
	\min_{u\in\mathbf U}\ 
	\psi_{\rho_r}(u;\theta,\lambda_r)
	+
	\delta_{\mathbf U}(u).
\end{equation*}
Since the slack ALM subproblem satisfies LICQ and SOSC at
\(y_{r+1}\), it has local quadratic growth on the identified active manifold:
for nearby feasible \(y=(u,z)\),
\begin{equation}
	\mathcal L_{\rho_r}(y,\lambda_r;\theta)
	\ge
	\mathcal L_{\rho_r}(y_{r+1},\lambda_r;\theta)
	+
	\alpha\|y-y_{r+1}\|^2
	\label{eq:app_slack_quadratic_growth}
\end{equation}
for some \(\alpha>0\). Taking \(z=z_{\rho_r}(u)\) and using
\(z_{r+1}=z_{\rho_r}(u_{r+1})\), we obtain
\begin{equation}
	\psi_{\rho_r}(u;\theta,\lambda_r)
	\ge
	\psi_{\rho_r}(u_{r+1};\theta,\lambda_r)
	+
	\alpha\|u-u_{r+1}\|^2
	\label{eq:app_reduced_quadratic_growth}
\end{equation}
for all nearby \(u\in\mathbf U\). Thus \(u_{r+1}\) is a strong local minimizer
of the reduced composite subproblem.

It remains to justify prox-regularity. By Assumption~\ref{ass:local_regular},
\(\mathbf U\) admits a local \(C^2\) active-manifold representation
\(b_{\mathbf U}(u)=0\), and LICQ holds on this manifold. Hence every nearby
normal vector \(v\in N_{\mathbf U}(u)\) admits the representation
\(v=\nabla b_{\mathbf U}(u)^\top\xi\), with
\(\|\xi\|\le c_0\|v\|\). For nearby \(u,\bar u\in\mathbf U\), Taylor
expansion of
\(b_{\mathbf U}(\bar u)=b_{\mathbf U}(u)=0\) gives
\begin{equation}
	\left\langle
	\nabla b_{\mathbf U}(u)^\top\xi,\bar u-u
	\right\rangle
	\le
	c_1\|\xi\|\,\|\bar u-u\|^2.
	\label{eq:app_active_manifold_taylor}
\end{equation}
Therefore,
\begin{equation}
	\langle v,\bar u-u\rangle
	\le
	c_0c_1\|v\|\,\|\bar u-u\|^2,
	\label{eq:prox_regular_ineq}
\end{equation}
which is the local prox-regularity inequality. Thus \(\mathbf U\) is
prox-regular around \(u_{r+1}\), equivalently
\(\delta_{\mathbf U}\) is locally prox-regular relative to the identified
active manifold.

Together with the local smoothness of \(\psi_{\rho_r}\), these properties
satisfy the local regularity assumptions of the \solver{PANDA} convergence result.
Hence, the \solver{PANDA} inner iteration is locally well posed and converges locally
to \(u_{r+1}\). This proves claim (iii), and the proof is complete.

\subsection{Sensitivity of the Final lapanda Subproblem}
\label{app:alm_sensitivity_derivation}

Recall that the reduced subproblem at the \(r\)-th ALM iteration is
\begin{equation}
	u_{r+1}(\theta)
	\in
	\underset{u\in\mathbf U}{\operatorname{argmin}}
	\psi_{\rho_r}(u;\theta,\lambda_r),
	\label{eq:app_final_reduced_subproblem}
\end{equation}
where
$
\psi_{\rho_r}(u;\theta,\lambda_r)
=
\ell(u,\theta)
+
\frac{\rho_r}{2}
\operatorname{dist}^2
\left(
G(u,\theta)+\rho_r^{-1}\lambda_r,
\mathcal D
\right).
\label{eq:app_final_reduced_objective}
$

As in Eq.~(\ref{eq:active_manifold_constraints}), let
\(a(u,\theta)=0\) collect the equality and locally active inequality
constraints, and let \(\phi(u)=0\) describe the identified active manifold
of \(\mathbf U\). At \(u_{r+1}\), define the shifted active-set multiplier
\(
\bar\eta_{r+1}
:=
\eta_r
+
\rho_r a(u_{r+1},\theta).
\)
Here, \(\eta_r\) contains the corresponding components of \(\lambda_r\).
The inactive inequality components do not enter the local reduced system
because their projections remain in the interior of \(\mathbb R_-\). Let
\(A_{r+1}:=\nabla_u a(u_{r+1},\theta)\) and
\(C_{r+1}:=\nabla_u\phi(u_{r+1})\), and let \(\mu_{r+1}\) be the
multiplier of \(\phi(u)=0\). The local KKT conditions are
\begin{equation}
	\begin{aligned}
		\nabla_u\ell(u_{r+1},\theta)
		+
		A_{r+1}^\top\bar\eta_{r+1}
		+
		C_{r+1}^\top\mu_{r+1}
		&=0,\\
		a(u_{r+1},\theta)
		-
		\rho_r^{-1}
		\left(
		\bar\eta_{r+1}-\eta_r
		\right)
		&=0,\\
		\phi(u_{r+1})
		&=0.
	\end{aligned}
	\label{eq:alm_reduced_active_kkt}
\end{equation}
The second condition is the definition of \(\bar\eta_{r+1}\) rearranged.
Throughout the differentiation, \(\eta_r\) and \(\rho_r\) are treated as
fixed.
Differentiating the stationarity condition with respect to \(\theta\) gives
\vspace{-1mm}
\begin{equation}
	H_{r+1}D_\theta u_{r+1}
	+
	A_{r+1}^\top D_\theta\bar\eta_{r+1}
	+
	C_{r+1}^\top D_\theta\mu_{r+1}
	+
	B_{r+1}
	=0,
	\label{eq:alm_sens_dF1}
\end{equation}
where
$
H_{r+1}
:=
\nabla_{uu}^2
\left(
\ell
+
\bar\eta_{r+1}^\top a
+
\mu_{r+1}^\top\phi
\right),
\quad
B_{r+1}
:=
\nabla_{u\theta}^2
\left(
\ell
+
\bar\eta_{r+1}^\top a
+
\mu_{r+1}^\top\phi
\right),
$
with all derivatives evaluated at
\((u_{r+1},\theta,\bar\eta_{r+1},\mu_{r+1})\).
Differentiating the shifted active-constraint relation gives
\vspace{-1mm}
\begin{equation}
	A_{r+1}D_\theta u_{r+1}
	-
	\rho_r^{-1}D_\theta\bar\eta_{r+1}
	+
	D_{r+1}
	=0,
	\label{eq:alm_sens_dF2}
\end{equation}
where
$
D_{r+1}
:=
\nabla_\theta a(u_{r+1},\theta).
$

Finally, since \(\mathbf U\) is independent of \(\theta\), differentiating
\(\phi(u_{r+1})=0\) gives
\begin{equation}
	C_{r+1}D_\theta u_{r+1}
	=
	0.
	\label{eq:alm_sens_dF3}
\end{equation}
Stacking Eqs.~(\ref{eq:alm_sens_dF1})--(\ref{eq:alm_sens_dF3}) yields
\begin{equation}
	\begin{bmatrix}
		H_{r+1} & A_{r+1}^\top & C_{r+1}^\top\\
		A_{r+1} & -\rho_r^{-1}I & 0\\
		C_{r+1} & 0 & 0
	\end{bmatrix}
	\begin{bmatrix}
		D_\theta u_{r+1}\\
		D_\theta\bar\eta_{r+1}\\
		D_\theta\mu_{r+1}
	\end{bmatrix}
	=
	-
	\begin{bmatrix}
		B_{r+1}\\
		D_{r+1}\\
		0
	\end{bmatrix}.
	\label{eq:lapanda_sensitivity_a}
\end{equation}

\subsection{Proof of Theorem~\ref{thm:sensitivity_alignment}}
\label{app:proof_sensitivity_alignment}

We prove Theorem~\ref{thm:sensitivity_alignment}. Let
\begin{equation}
	d^\star
	:=
	\begin{bmatrix}
		D_\theta u^\star\\
		D_\theta\eta^\star\\
		D_\theta\mu^\star
	\end{bmatrix},
	\qquad
	d_{r+1}
	:=
	\begin{bmatrix}
		D_\theta u_{r+1}\\
		D_\theta\bar\eta_{r+1}\\
		D_\theta\mu_{r+1}
	\end{bmatrix}.
	\label{eq:app_sensitivity_vectors}
\end{equation}
The original active-set KKT sensitivity system and the reduced \solver{lapanda}
sensitivity system can be written compactly as
\begin{equation}
	Kd^\star=-\beta,
	\qquad
	K_{r+1}d_{r+1}=-\beta_{r+1},
	\label{eq:app_compact_sensitivity_systems}
\end{equation}
where
\[
K
=
\begin{bmatrix}
	H & A^\top & C^\top\\
	A & 0 & 0\\
	C & 0 & 0
\end{bmatrix},
\qquad
K_{r+1}
=
\begin{bmatrix}
	H_{r+1} & A_{r+1}^\top & C_{r+1}^\top\\
	A_{r+1} & -\rho_r^{-1}I & 0\\
	C_{r+1} & 0 & 0
\end{bmatrix},
\label{eq:app_sensitivity_matrices}
\]
and
\[
\beta
=
\begin{bmatrix}
	B\\
		D\\
	0
\end{bmatrix},
\qquad
\beta_{r+1}
=
\begin{bmatrix}
	B_{r+1}\\
		D_{r+1}\\
	0
\end{bmatrix}.
\label{eq:app_sensitivity_rhs}
\]
The third block is zero because the active manifold
\(\phi(u)=0\) of the simple set \(\mathbf U\) is independent of
\(\theta\). If \(\mathbf U\) depends on \(\theta\), the corresponding
blocks are replaced by \(E\) and \(E_{r+1}\).

Under Assumption~\ref{ass:local_regular}, strict complementarity implies local
active-set identification. Hence, for all sufficiently advanced ALM iterates,
the active manifolds used in
Eqs.~(\ref{eq:original_kkt_sensitivity}) and
(\ref{eq:lapanda_sensitivity_a}) coincide locally. Moreover, LICQ and SOSC
imply that the original active-set KKT matrix \(K\) is nonsingular.

We next compare the two linear systems. By Theorem~\ref{thm:lapanda_wellposed}, the evaluated quantities
\((u_{r+1},\bar\eta_{r+1},\mu_{r+1})\) remain in a neighborhood of
\((u^\star,\eta^\star,\mu^\star)\). Since the active functions are
\(C^{2,1}\), the Jacobian and Lagrangian-Hessian blocks are locally Lipschitz
in the primal-dual variables. Therefore,
\begin{equation}
	H_{r+1}-H=O(\epsilon_{r+1}),
	\qquad
	A_{r+1}-A=O(\epsilon_{r+1}),
	\qquad
	C_{r+1}-C=O(\epsilon_{r+1}),
	\label{eq:app_sensitivity_block_errors}
\end{equation}
where
$
\epsilon_{r+1}
:=
\|u_{r+1}-u^\star\|
+
\|\bar\eta_{r+1}-\eta^\star\|
+
\|\mu_{r+1}-\mu^\star\|.
$

The finite-penalty block is the only structural difference between the two
systems. Therefore,
\begin{equation}
	K_{r+1}
	=
		K+\Delta_{\rho_r}+\mathcal E_{r+1},
	\qquad
	\Delta_{\rho_r}
	=
	\begin{bmatrix}
		0 & 0 & 0\\
		0 & -\rho_r^{-1}I & 0\\
		0 & 0 & 0
	\end{bmatrix},
	\qquad
		\|\mathcal E_{r+1}\|=O(\epsilon_{r+1}).
	\label{eq:app_Kr_perturbation}
\end{equation}
\begin{subequations}
	\label{eq:app_error}
	Consequently,
	\begin{equation}
		\|K_{r+1}-K\|
		=
		O(\rho_r^{-1})+O(\epsilon_{r+1}).
		\label{eq:app_Kr_error}
	\end{equation}
	The same \(C^{2,1}\) regularity yields
	\begin{equation}
		\|\beta_{r+1}-\beta\|
		=
		O(\epsilon_{r+1}).
		\label{eq:app_rhs_error}
	\end{equation}
	
Since \(K\) is nonsingular, the perturbation estimate in
Eq.~(\ref{eq:app_Kr_error}) implies that \(K_{r+1}\) is nonsingular whenever
	\(\rho_r\) is sufficiently large and \(\epsilon_{r+1}\) is sufficiently small.
	Moreover, its inverse is locally uniformly bounded: there exists
	\(\kappa_K>0\) such that
	\begin{equation}
		\|K_{r+1}^{-1}\|
		\le
		\kappa_K.
		\label{eq:app_Kr_inverse_bound}
	\end{equation}
\end{subequations}

Subtracting the two systems in
Eq.~(\ref{eq:app_compact_sensitivity_systems}) gives
\begin{equation*}
	K_{r+1}(d_{r+1}-d^\star)
	=
	-(\beta_{r+1}-\beta)
	-
	(K_{r+1}-K)d^\star.
	\label{eq:app_sensitivity_system_difference}
\end{equation*}
Hence,
\begin{equation*}
	d_{r+1}-d^\star
	=
	-K_{r+1}^{-1}(\beta_{r+1}-\beta)
	-
	K_{r+1}^{-1}(K_{r+1}-K)d^\star.
	\label{eq:app_sensitivity_solution_difference}
\end{equation*}
Using Eqs.~(\ref{eq:app_Kr_error}), (\ref{eq:app_rhs_error}), and
(\ref{eq:app_Kr_inverse_bound}), we obtain
\begin{equation}
	\|d_{r+1}-d^\star\|
	\le
	\kappa_K
	\left(
	\|\beta_{r+1}-\beta\|
	+
	\|K_{r+1}-K\|\,\|d^\star\|
	\right)
	=
	O(\rho_r^{-1})+O(\epsilon_{r+1}).
	\label{eq:app_full_sensitivity_error}
\end{equation}
Taking the first block yields
\begin{equation}
	D_\theta u_{r+1}
	=
	D_\theta u^\star
	+
	O(\rho_r^{-1})
	+
	O(\epsilon_{r+1}).
	\label{eq:app_primal_sensitivity_alignment}
\end{equation}

It remains to show that \(\epsilon_{r+1}\to0\) as the ALM outer iteration
converges. The local ALM convergence result in
Theorem~\ref{thm:lapanda_wellposed} gives
\begin{equation}
	u_{r+1}\to u^\star,
	\qquad
	\bar\eta_{r+1}\to\eta^\star.
	\label{eq:app_primal_dual_convergence}
\end{equation}

The multiplier \(\mu_{r+1}\) is associated with the identified active manifold of
\(\mathbf U\). The local stationarity condition of the final reduced
subproblem is
\begin{equation}
	\nabla_u\ell(u_{r+1},\theta)
	+
	A_{r+1}^\top\bar\eta_{r+1}
	+
	C_{r+1}^\top\mu_{r+1}
	=
	0.
	\label{eq:app_final_subproblem_stationarity}
\end{equation}
By LICQ, the active-set multiplier is locally unique and bounded. Hence, \(\{\mu_{r+1}\}\) is bounded and admits convergent subsequences. Let \(\bar\mu\) be any accumulation point. Passing to the limit in Eq.~(\ref{eq:app_final_subproblem_stationarity}) gives
\begin{equation}
	\nabla_u\ell(u^\star,\theta)
	+
	A^\top\eta^\star
	+
	C^\top\bar\mu
	=
	0.
	\label{eq:app_limit_stationarity}
\end{equation}
The original active-set KKT stationarity condition is
\begin{equation}
	\nabla_u\ell(u^\star,\theta)
	+
	A^\top\eta^\star
	+
	C^\top\mu^\star
	=
	0.
	\label{eq:app_original_stationarity_repeated}
\end{equation}
Subtracting Eq.~(\ref{eq:app_original_stationarity_repeated}) from
Eq.~(\ref{eq:app_limit_stationarity}) gives
\(C^\top(\bar\mu-\mu^\star)=0\). Since the active constraint gradients are
linearly independent, \(C\) has full row rank, and therefore
\(\bar\mu=\mu^\star\). Thus every accumulation point equals \(\mu^\star\), which implies
\begin{equation}
	\mu_{r+1}\to\mu^\star
	\qquad\text{and}\qquad
	\epsilon_{r+1}\to0.
	\label{eq:app_simple_multiplier_convergence}
\end{equation}

Therefore, as the ALM outer iteration converges, the
\(O(\epsilon_{r+1})\) term in
Eq.~(\ref{eq:app_primal_sensitivity_alignment}) vanishes. If the \solver{PANDA} inner
subproblems are solved inexactly, the same conclusion holds provided that their
stationarity residuals converge to zero. If, in addition,
\(\rho_r\to\infty\), then the finite-penalty term
\(O(\rho_r^{-1})\) also vanishes, and hence
\begin{equation}
	D_\theta u_{r+1}
	\longrightarrow
	D_\theta u^\star.
	\label{eq:app_sensitivity_convergence}
\end{equation}

\subsection{Matrix-Free Operators}
\label{app:matrix_free_operators}

We derive the matrix-free backward operators for the final reduced
\solver{lapanda} subproblem from the residual equation
\begin{equation}
	R_{\gamma,r}(u,\theta)
	=
	\frac{1}{\gamma}
	\left[
	u-
	\Pi_{\mathbf U}
	\left(
	u-\gamma\nabla_u\psi_{\rho_r}(u;\theta,\lambda_r)
	\right)
	\right].
	\label{eq:mf_residual}
\end{equation}
At the returned solution \(u_{r+1}\), we have
\(R_{\gamma,r}(u_{r+1},\theta)\approx0\). Instead of explicitly forming
\(D_\theta u_{r+1}\), the backward pass solves the adjoint system
\begin{equation}
	\left[
	\nabla_uR_{\gamma,r}(u_{r+1},\theta)
	\right]^\top
	v_{\mathrm{adj}}
	=
		\nabla_u\mathcal L_{\mathrm{out}}(u_{r+1},\theta)^\top,
	\label{eq:mf_adjoint}
\end{equation}
and then computes
\begin{equation}
		\nabla_\theta\mathcal L_{\mathrm{out}}
		=
		\nabla_\theta\mathcal L_{\mathrm{out}}(u_{r+1},\theta)
	-
	v_{\mathrm{adj}}^\top
	\nabla_\theta R_{\gamma,r}(u_{r+1},\theta).
	\label{eq:mf_outer_grad}
\end{equation}
Krylov iterative methods solve the adjoint linear system using only
matrix--vector products with its coefficient operator. Therefore, the backward
pass only requires the two matrix-free products
\[
\left[\nabla_uR_{\gamma,r}\right]^\top v_{\mathrm{adj}},
\qquad
v_{\mathrm{adj}}^\top\nabla_\theta R_{\gamma,r}.
\]

Let
\[
y_{\gamma,r+1}
=
u_{r+1}
-
\gamma\nabla_u\psi_{\rho_r}
(u_{r+1};\theta,\lambda_r),
\qquad
P_{U,r+1}
=
D\Pi_{\mathbf U}(y_{\gamma,r+1}).
\]
When \(\mathbf U\) is a box, \(P_{U,r+1}\) is a diagonal active-set mask that
retains the free components and removes the active components. Since
\(\mathbf U\) is independent of \(\theta\), differentiating
Eq.~(\ref{eq:mf_residual}) gives
\begin{equation}
	\nabla_uR_{\gamma,r}
	=
	\frac{1}{\gamma}(I-P_{U,r+1})
	+
	P_{U,r+1}\nabla_{uu}^2\psi_{\rho_r},
	\qquad
	\nabla_\theta R_{\gamma,r}
	=
	P_{U,r+1}\nabla_{u\theta}^2\psi_{\rho_r}.
	\label{eq:mf_residual_derivatives}
\end{equation}
Therefore, the two products required by
Eqs.~(\ref{eq:mf_adjoint}) and (\ref{eq:mf_outer_grad}) are
\begin{equation}
	\left[
	\nabla_uR_{\gamma,r}
	\right]^\top v_{\mathrm{adj}}
	=
	\frac{1}{\gamma}
	(I-P_{U,r+1})^\top v_{\mathrm{adj}}
	+
	\nabla_{uu}^2\psi_{\rho_r}
	\left(
	P_{U,r+1}^\top v_{\mathrm{adj}}
	\right),
	\label{eq:mf_residual_transpose_product}
\end{equation}
and
\begin{equation}
	v_{\mathrm{adj}}^\top\nabla_\theta R_{\gamma,r}
	=
	\left(
	P_{U,r+1}^\top v_{\mathrm{adj}}
	\right)^\top
	\nabla_{u\theta}^2\psi_{\rho_r}.
	\label{eq:mf_residual_theta_product}
\end{equation}
Hence, it remains to evaluate the smooth-part Hessian-vector product
\(\nabla_{uu}^2\psi_{\rho_r}v\) and the mixed vector-Jacobian product
\(q^\top\nabla_{u\theta}^2\psi_{\rho_r}\).

Recall the reduced ALM objective in
Eq.~(\ref{eq:reduced_alm_objective}). At \(u_{r+1}\), define
\[
w_{r+1}
=
G(u_{r+1},\theta)+\rho_r^{-1}\lambda_r,
\qquad
z_{r+1}
=
\Pi_{\mathcal D}(w_{r+1}),
\qquad
	\bar\lambda_{r+1}
	=
	\rho_r(w_{r+1}-z_{r+1}).
\]
Also define
\[
J_{G,r+1}
=
\nabla_uG(u_{r+1},\theta),
\qquad
P_{D,r+1}
=
D\Pi_{\mathcal D}(w_{r+1}),
\qquad
S_{D,r+1}
=
I-P_{D,r+1}.
\]
For
\(\mathcal D=\{0\}^{n_c}\times\mathbb R_-^{n_h}\), the mask \(S_{D,r+1}\)
retains the equality components and the locally active inequality components,
while removing the inactive inequality components.

The gradient of the reduced smooth objective is
\begin{equation}
	\nabla_u\psi_{\rho_r}
	(u_{r+1};\theta,\lambda_r)
	=
	\nabla_u\ell(u_{r+1},\theta)
	+
	J_{G,r+1}^\top\bar\lambda_{r+1}.
	\label{eq:mf_smooth_grad}
\end{equation}
Since
$
D_u\bar\lambda_{r+1}[v]
=
\rho_rS_{D,r+1}J_{G,r+1}v,
$
the Hessian-vector product satisfies, for any vector \(v\),
\begin{equation}
	\nabla_{uu}^2\psi_{\rho_r}v
	=
	\nabla_{uu}^2
	\left(
	\ell+\bar\lambda_{r+1}^\top G
	\right)v
	+
	\rho_r
	J_{G,r+1}^\top S_{D,r+1}J_{G,r+1}v.
	\label{eq:mf_smooth_hvp}
\end{equation}
In the first term, \(\bar\lambda_{r+1}\) is held fixed. Thus, this term is a
Lagrangian Hessian-vector product. The second term requires only one JVP
through \(G\), one application of the active-set mask \(S_{D,r+1}\), and one VJP
through \(G\).

During differentiation of the final ALM subproblem, the incoming ALM
multiplier \(\lambda_r\) and the penalty parameter \(\rho_r\) are treated as
fixed. Since
$
D_\theta\bar\lambda_{r+1}[\delta\theta]
=
\rho_rS_{D,r+1}
\nabla_\theta G(u_{r+1},\theta)\delta\theta,
$
the mixed product satisfies
\begin{equation}
	\nabla_{u\theta}^2\psi_{\rho_r}\delta\theta
	=
	\nabla_{u\theta}^2
	\left(
	\ell+\bar\lambda_{r+1}^\top G
	\right)\delta\theta
	+
	\rho_r
	J_{G,r+1}^\top S_{D,r+1}
	\nabla_\theta G(u_{r+1},\theta)\delta\theta.
	\label{eq:mf_smooth_cross}
\end{equation}
Equivalently, for any adjoint vector \(q\),
\begin{equation}
	q^\top\nabla_{u\theta}^2\psi_{\rho_r}
	=
	q^\top
	\nabla_{u\theta}^2
	\left(
	\ell+\bar\lambda_{r+1}^\top G
	\right)
	+
	\rho_r
	\left(
	S_{D,r+1}J_{G,r+1}q
	\right)^\top
	\nabla_\theta G(u_{r+1},\theta).
	\label{eq:mf_smooth_cross_vjp}
\end{equation}

Let \(q=P_{U,r+1}^\top v_{\mathrm{adj}}\). Substituting
Eq.~(\ref{eq:mf_smooth_hvp}) into
Eq.~(\ref{eq:mf_residual_transpose_product}), the adjoint operator used in the
Krylov solver is
\begin{equation}
	\left[
	\nabla_uR_{\gamma,r}
	\right]^\top v_{\mathrm{adj}}
	=
	\frac{1}{\gamma}
	(I-P_{U,r+1})^\top v_{\mathrm{adj}}
	+
	\nabla_{uu}^2
	\left(
	\ell+\bar\lambda_{r+1}^\top G
	\right)q
	+
	\rho_r
	J_{G,r+1}^\top S_{D,r+1}J_{G,r+1}q.
	\label{eq:mf_final_adjoint_operator}
\end{equation}
The smooth Hessian term and the projected-constraint term
\(J_{G,r+1}^\top S_{D,r+1}J_{G,r+1}\) are symmetric under the local regularity
assumptions. However, their composition with the projection derivative
\(P_{U,r+1}\), through \(q=P_{U,r+1}^\top v_{\mathrm{adj}}\), does not generally preserve
the symmetry of the complete residual operator. Hence, GMRES provides the
general Krylov solver. For the common case in which \(\mathbf U\) consists
of box constraints and the active set is locally identified,
\(P_{U,r+1}\) is a symmetric diagonal projector. Its range
\(\mathcal V_{r+1}:=\operatorname{range}(P_{U,r+1})\) is the tangent space of
the identified box face. Restricting the adjoint system to this subspace gives
the projected operator
\(P_{U,r+1}\nabla_{uu}^2\psi_{\rho_r}P_{U,r+1}\). For every nonzero
\(d\in\mathcal V_{r+1}\), Eq.~(\ref{eq:alm_pd_reference}), together with the
slack elimination leading to Eq.~(\ref{eq:app_reduced_quadratic_growth}), gives
\begin{equation}
	d^\top P_{U,r+1}\nabla_{uu}^2\psi_{\rho_r}P_{U,r+1}d
	=
	d^\top\nabla_{uu}^2\psi_{\rho_r}d
	>
	0.
	\label{eq:mf_projected_operator_pd}
\end{equation}
Thus, locally under SOSC and a sufficiently large ALM penalty, the restricted
operator is symmetric positive definite, so we use CG by default. If
nonpositive curvature or numerical failure is detected, the implementation
automatically falls back to MINRES.

Substituting
Eq.~(\ref{eq:mf_smooth_cross_vjp}) into
Eq.~(\ref{eq:mf_residual_theta_product}), the parameter-side VJP is
\begin{equation}
	v_{\mathrm{adj}}^\top\nabla_\theta R_{\gamma,r}
	=
	q^\top
	\nabla_{u\theta}^2
	\left(
	\ell+\bar\lambda_{r+1}^\top G
	\right)
	+
	\rho_r
	\left(
	S_{D,r+1}J_{G,r+1}q
	\right)^\top
	\nabla_\theta G(u_{r+1},\theta).
	\label{eq:mf_final_theta_vjp}
\end{equation}
Eqs.~(\ref{eq:mf_final_adjoint_operator}) and
(\ref{eq:mf_final_theta_vjp}) give the final matrix-free backward operators.
Their evaluation requires only HVPs, JVPs, VJPs, and active-set masks
associated with \(\Pi_{\mathbf U}\) and \(\Pi_{\mathcal D}\), without
explicitly constructing Hessian, Jacobian, or KKT matrices.
\vspace{-3mm}
\section{Experiments Details}
\vspace{-2mm}
\subsection{Constrained Rosenbrock Benchmark}
\vspace{-2mm}
\label{app:rosenbrock_experiment_details}
This section examines the sensitivity-accuracy behavior characterized by
Theorem~\ref{thm:sensitivity_alignment} and the practical strategies used to
control it, thereby explaining the construction, settings, and
accuracy-matching protocol of Table~\ref{tab:lapanda_casadi_benchmark}.

\noindent\textbf{Experimental setup.}
This benchmark and the constrained OCP imitation experiments
were conducted on a workstation equipped with an Intel Core i5-12600KF CPU@3.70 GHz, 32 GB of RAM, and an NVIDIA GeForce RTX 4060Ti GPU. The learnable parameter is \(\theta=(\theta_0,\theta_1,\theta_2,\theta_3)\), with nominal value
\((10,1,10^{-3},1.35)\), and the outer sensitivity loss is
\(\mathcal L_{\mathrm{out}}(x^\star)=\frac{1}{2}\|x^\star-x_{\mathrm{tar}}\|_2^2\), where
\(x_{\mathrm{tar}}\) is the sinusoidal target used in the benchmark. The solver settings are summarized in
Table~\ref{tab:app_rosenbrock_settings}.

\vspace{-4mm}
\begin{table}[h]
	\centering
	\caption{Solver settings for the constrained Rosenbrock benchmark.}
	\label{tab:app_rosenbrock_settings}
	\begingroup
	\renewcommand{\arraystretch}{1.1}
	\resizebox{0.6\linewidth}{!}{%
		\begin{tabular}{lc}
			\toprule
			Quantity & Value \\
			\midrule
			CasADi forward solver & IPOPT \\
			CasADi sensitivity pipeline & SQP/qpOASES + sparse QR \\
			Forward ALM/IPOPT tolerance & \(10^{-3}\) \\
			Krylov tolerance & \(10^{-2}\) \\
			lapanda/Explicit Krylov method & CG/MINRES \\
			Maximum Krylov iterations & \(200\) \\
			Maximum PANDA iterations & \(4000\) \\
			Maximum ALM iterations & \(20\) \\
			Initial penalty \(\rho_0\) & \(2\) \\
			Penalty update factor & \(10\) \\
			\bottomrule
		\end{tabular}%
	}
	\endgroup
\end{table}
\vspace{-2mm}

Following the official \solver{CasADi} sensitivity workflow, we first solved the original NLP with IPOPT and used the converged primal solution to initialize a differentiable \texttt{sqpmethod} solver. Since this initialization already satisfied the prescribed tolerance, the SQP stage terminated after one iteration. During this iteration, \texttt{qpOASES} computed the information required for active-set identification. \solver{CasADi} then constructed the differentiated KKT system using the exact Lagrangian Hessian and constraint Jacobian and solved the adjoint system with its default sparse-QR sensitivity solver. We also tested \texttt{qrqp} as the QP backend of \texttt{sqpmethod}, but it occasionally failed on larger instances. For the problem sizes successfully solved by both backends, their runtime and memory overheads were similar. We therefore report results obtained with the \texttt{qpOASES}-based configuration.

\noindent\textbf{Sensitivity reference and error metrics.}
To obtain a reference sensitivity for each solution returned by \solver{lapanda}, we construct the KKT sensitivity
system of the original problem and solve it to high accuracy. The
relative linear residual is below \(10^{-15}\) in all reported cases.
Let \(g=\nabla_\theta\mathcal L_{\mathrm{out}}^{\mathrm{lapanda}}\) and let
\(g_{\mathrm{ref}}\) denote the reference gradient. We report both relative
error and cosine similarity:
\begin{equation}
	\mathrm{Err}_{\nabla}
	=
	\frac{\|g-g_{\mathrm{ref}}\|_2}{\|g_{\mathrm{ref}}\|_2},
	\qquad
	\mathrm{Cos}_{\nabla}
	=
	\frac{g^\top g_{\mathrm{ref}}}
	{\|g\|_2\|g_{\mathrm{ref}}\|_2}.
	\label{eq:app_gradient_metrics}
\end{equation}
The former measures the full gradient discrepancy, whereas the latter isolates
directional agreement. Constraint violation is measured by
\(\max_i[g_i(x,\theta)]_+\).

\noindent\textbf{Relationship between penalty and sensitivity error.}
The analysis in Section~4.2 gives the error order
\(O(\rho_r^{-1})+O(\epsilon_{r+1})\). To examine this relation, we solve one
ALM subproblem for the representative \(n=200\) instance at each fixed
penalty. Table~\ref{tab:app_penalty_sweep} shows that increasing \(\rho\)
improves feasibility and gradient accuracy, while making the adjoint system
more expensive.
\vspace{-4mm}

\begin{table}[h]
	\centering
	\caption{Penalty sweep for the \(n=200\) constrained Rosenbrock diagnostic.}
	\label{tab:app_penalty_sweep}
	\renewcommand{\arraystretch}{1.1}
	\begin{tabular}{ccccc}
		\toprule
		Penalty \(\rho\)
		& Relative error \textcolor{green!60!black}{\(\searrow\)}
		& Cosine similarity
		& Violation \textcolor{green!60!black}{\(\searrow\)}
		& Backward time (ms) \textcolor{red!70!black}{\(\nearrow\)} \\
		\midrule
		\(1\) & \(5.59\times10^{-1}\) & \(0.988404\) & \(2.47\times10^{-1}\) & \(0.15\) \\
		\(3\) & \(2.91\times10^{-1}\) & \(0.998718\) & \(1.28\times10^{-1}\) & \(0.15\) \\
		\(10\) & \(1.14\times10^{-1}\) & \(0.999793\) & \(4.67\times10^{-2}\) & \(0.18\) \\
		\(30\) & \(4.15\times10^{-2}\) & \(0.999967\) & \(1.84\times10^{-2}\) & \(0.22\) \\
		\(100\) & \(1.34\times10^{-2}\) & \(0.999994\) & \(6.74\times10^{-3}\) & \(0.31\) \\
		\(300\) & \(4.80\times10^{-3}\) & \(0.999999\) & \(2.32\times10^{-3}\) & \(0.53\) \\
		\(1000\) & \(6.93\times10^{-4}\) & \(1.000000\) & \(7.02\times10^{-4}\) & \(0.76\) \\
		\(3000\) & \(5.64\times10^{-4}\) & \(1.000000\) & \(2.34\times10^{-4}\) & \(0.99\) \\
		\bottomrule
	\end{tabular}
\end{table}
\vspace{-2mm}
\noindent\textbf{Balanced strategy.}
Our default strategy directly reuses the primal solution, multiplier estimate,
and penalty returned by the converged forward ALM solve. The backward system
therefore corresponds to the same final subproblem, with no additional tuning
or forward optimization. Table~\ref{tab:app_converged_penalty} reports means
over ten instances at each problem size. Although the relative error is about
\(5\%\)--\(10\%\), the cosine similarity consistently exceeds \(0.9999\).
Thus, the finite-penalty approximation mainly affects the gradient magnitude,
while preserving its descent direction to high accuracy. This balanced
strategy retains the millisecond-scale backward pass and is used in the later
learning experiments, where it produces loss trajectories comparable to the
baselines.

\vspace{-4mm}
\begin{table}[h]
	\centering
	\caption{Balanced sensitivity computation using the forward-converged penalty.}
	\label{tab:app_converged_penalty}
	\begin{tabular}{cccccc}
		\toprule
		\(n\) & Final penalty & Relative error & Cosine similarity & Violation & Backward time (ms) \\
		\midrule
		\(100\) & \(4.67\times10^2\) & \(7.08\times10^{-2}\) & \(0.999925\) & \(7.99\times10^{-4}\) & \(0.15\) \\
		\(200\) & \(4.59\times10^2\) & \(9.42\times10^{-2}\) & \(0.999910\) & \(5.75\times10^{-4}\) & \(0.25\) \\
		\(500\) & \(3.24\times10^2\) & \(9.89\times10^{-2}\) & \(0.999922\) & \(6.18\times10^{-4}\) & \(0.62\) \\
		\(1000\) & \(4.53\times10^2\) & \(9.87\times10^{-2}\) & \(0.999932\) & \(6.86\times10^{-4}\) & \(1.54\) \\
		\bottomrule
	\end{tabular}
\end{table}
\vspace{-2mm}
\noindent\textbf{Post-forward penalty refinement.}
When higher relative accuracy is required, we consider two post-forward
refinement strategies. \emph{Direct backward refinement} retains the converged
primal-dual estimate and uses an enlarged penalty only in the backward
operators. It requires no additional forward optimization, but the modified
backward system is not exactly aligned with the subproblem solved in the
forward pass. \emph{Aligned subproblem refinement} instead increases the
penalty, re-solves the final ALM subproblem from the converged state, and then
differentiates the refined subproblem. This preserves the alignment in
Theorem~\ref{thm:sensitivity_alignment}, at the cost of one additional warm-started
subproblem solve. Table~\ref{tab:app_penalty_refinement} compares the two
strategies using \(10\rho\) on the first instance at each problem size.

\vspace{-4mm}
\begin{table}[H]
	\centering
	\caption{Comparison of post-forward penalty-refinement strategies.}
	\label{tab:app_penalty_refinement}
	\setlength{\tabcolsep}{3.5pt}
	\resizebox{\linewidth}{!}{
		\begin{tabular}{c|ccc|ccc|cccc}
			\toprule
			\multirow{2}{*}{\(n\)}
			& \multicolumn{3}{c|}{Forward-converged}
			& \multicolumn{3}{c|}{Direct backward refinement}
			& \multicolumn{4}{c}{Aligned subproblem refinement} \\
			\cmidrule(lr){2-4}
			\cmidrule(lr){5-7}
			\cmidrule(lr){8-11}
			& \(\rho\) & Grad. err. & Bwd. (ms)
			& \(\rho\) & Grad. err. & Bwd. (ms)
			& \(\rho\) & Grad. err. & Refine solve (ms) & Bwd. (ms) \\
			\midrule
		\(100\)
		& \(4.12\times10^2\) & \(7.07\times10^{-2}\) & \(0.15\)
		& \(4.12\times10^3\) & \(1.13\times10^{-2}\) & \(0.26\)
		& \(4.12\times10^3\) & \(9.29\times10^{-3}\) & \(3.71\) & \(0.27\) \\
		\(200\)
		& \(2.00\times10^2\) & \(9.54\times10^{-2}\) & \(0.23\)
		& \(2.00\times10^3\) & \(2.42\times10^{-2}\) & \(0.43\)
		& \(2.00\times10^3\) & \(1.37\times10^{-2}\) & \(4.17\) & \(0.44\) \\
		\(500\)
		& \(2.27\times10^2\) & \(9.97\times10^{-2}\) & \(0.58\)
		& \(2.27\times10^3\) & \(1.31\times10^{-2}\) & \(1.16\)
		& \(2.27\times10^3\) & \(1.31\times10^{-2}\) & \(11.30\) & \(1.20\) \\
		\(1000\)
		& \(2.36\times10^2\) & \(9.96\times10^{-2}\) & \(1.50\)
		& \(2.36\times10^3\) & \(1.24\times10^{-2}\) & \(3.04\)
		& \(2.36\times10^3\) & \(1.25\times10^{-2}\) & \(29.46\) & \(2.97\) \\
			\bottomrule
		\end{tabular}
	}
\end{table}

\vspace{-4mm}
As shown in Table~\ref{tab:app_penalty_refinement},
aligned refinement reduces the relative error below \(2\%\) at every
problem size. We therefore use it for the matched-accuracy comparison in the
main text. Direct refinement also improves accuracy without an additional
forward solve, but its backward system is not aligned with the
forward solution and therefore lacks the same theoretical justification.

These results distinguish two use cases. The forward-converged penalty is the
balanced default: its gradient direction is already highly accurate and no
extra computation is required. When higher relative gradient accuracy is
required, aligned refinement provides the principled option used in the main
benchmark. Direct refinement remains available as a low-cost heuristic knob.

It is worth noting that the subsequent imitation-learning experiments use the
forward-converged penalty by default. Although the resulting gradients have
moderate relative error, their high cosine similarity with the reference
gradient suggests that the update direction remains sufficiently accurate for
the learning tasks considered here. Such approximate gradients can be
substantially cheaper to compute, and some baselines likewise employ
approximate sensitivity methods. The broadly similar loss trends observed in
the subsequent experiments provide empirical support for this strategy.

\noindent\textbf{Matched-accuracy comparison and explicit-KKT ablation.}
Table~\ref{tab:lapanda_casadi_benchmark} uses ten instances at each problem
size. Gradient accuracy is evaluated by Eq.~(\ref{eq:app_gradient_metrics})
against the same high-accuracy original-NLP KKT gradient at the aligned
\solver{lapanda} point for each instance. For all three methods,
the timings are averaged over five repetitions on each of the same ten
instances, following an untimed warm-up at each problem size. For
\solver{lapanda}, we apply aligned
refinement with \(10\rho\). Its reported forward time contains only the original
ALM solve, whereas its backward time includes the warm-started final subproblem
refinement and the subsequent matrix-free adjoint solve.
\vspace{-4mm}
\begin{table}[h]
	\centering
	\caption{Mean (maximum) relative gradient error (\%) under the matched-accuracy protocol.}
	\label{tab:app_matched_accuracy}
	\setlength{\tabcolsep}{8pt}
	\begin{tabular}{cccc}
		\toprule
		\(n\) & lapanda & Explicit KKT & CasADi \\
		\midrule
		\(100\)  & \(0.95\;(1.16)\) & \(1.09\;(2.32)\) & \(0.02\;(0.03)\) \\
		\(200\)  & \(1.35\;(1.46)\) & \(0.68\;(0.74)\) & \(0.01\;(0.02)\) \\
		\(500\)  & \(1.30\;(1.35)\) & \(0.23\;(0.23)\) & \(0.43\;(0.43)\) \\
		\(1000\) & \(1.24\;(1.26)\) & \(0.33\;(0.33)\) & \(<0.01\;(<0.01)\) \\
		\bottomrule
	\end{tabular}
\end{table}

\vspace{-2mm}
The explicit-KKT baseline shares the unrefined \solver{lapanda} forward result,
explicitly assembles the KKT matrix of the original NLP, and solves the
adjoint system with MINRES through the same C-based Krylov interface and
tolerance. MINRES is used because the symmetric saddle-point KKT matrix is
generally indefinite. It therefore provides a matrix-based sensitivity baseline
under the same base forward solve.

Memory overhead is evaluated using the resident set size (RSS) of the solver
process. We consider both the memory increase during solver construction and
the temporary memory used during solution. Solver construction mainly includes
problem initialization, computational-graph construction, and operator
generation, whereas temporary memory accounts for intermediate variables and
workspaces allocated during solving. The peak temporary increment typically
occurs within the first few solves, since later calls can reuse previously
allocated memory. We therefore report
\begin{equation}
	\text{total memory}
	=
	\text{build RSS peak increment}
	+
	\text{peak solve RSS increment}.
	\label{eq:app_total_memory_measurement}
\end{equation}

\vspace{-4mm}
\subsection{Constrained OCP Imitation Experiments}
\label{app:ocp_experiment_details}
\textbf{Common setup.}
All OCP experiments use a prediction horizon of \(N=20\) and a sampling
interval of \(\Delta t=0.05\). We adopt the single-shooting transcription
described in Appendix~\ref{app:ocp_transcription}. Thus, the optimization
variable is the stacked control sequence
$
u=(u_0^\top,\ldots,u_{N-1}^\top)^\top,
$
whereas the state trajectory is obtained by recursively rolling out the
dynamics. The teacher trajectory is generated using the target parameter
\(\theta_{\mathrm{true}}\), and the learner minimizes the imitation loss in
Eq.~(\ref{eq:imit_loss}).

\noindent\textbf{CartPole.}
For the CartPole task, the system state is defined as
$x_k=(p_k,\dot p_k,\phi_k,\dot\phi_k)$, where $p_k$ and $\dot p_k$
denote the cart position and velocity, while $\phi_k$ and $\dot\phi_k$
denote the pole angle and angular velocity, respectively. The control input,
denoted by $F_k$, is the horizontal force. Using forward
Euler discretization, the CartPole dynamics are given by
\begin{equation}
	\begin{aligned}
		\ddot{\phi}_k
		&=
		\tfrac{
			g\sin\phi_k
			-
			\tfrac{
				\cos\phi_k
				\left(
				F_k
				+
				m_p l\dot{\phi}_k^2\sin\phi_k
				\right)
			}{
				m_c+m_p
			}
		}{
			l\left(
			\tfrac{4}{3}
			-
			\tfrac{
				m_p\cos^2\phi_k
			}{
				m_c+m_p
			}
			\right)
		},
		\\
		\ddot{p}_k
		&=
		\tfrac{
			F_k
			+
			m_p l
			\left(
			\dot{\phi}_k^2\sin\phi_k
			-
			\ddot{\phi}_k\cos\phi_k
			\right)
		}{
			m_c+m_p
		},
		\\
		p_{k+1}
		&=
		p_k+\Delta t\,\dot{p}_k,
		\\
		\dot{p}_{k+1}
		&=
		\dot{p}_k+\Delta t\,\ddot{p}_k,
		\\
		\phi_{k+1}
		&=
		\phi_k+\Delta t\,\dot{\phi}_k,
		\\
		\dot{\phi}_{k+1}
		&=
		\dot{\phi}_k+\Delta t\,\ddot{\phi}_k.
	\end{aligned}
	\label{eq:app_cartpole_dynamics}
\end{equation}
where \(m_c\) and \(m_p\) denote the masses of the cart and pole,
respectively, \(l\) is the pole half-length, \(g\) is the gravitational
acceleration, and \(\Delta t\) is the discretization step.

We consider a finite-horizon planning problem that steers the CartPole system
from an initial state $\xi$ toward the upright equilibrium
$x_{\mathrm{tar}}=(p_{\mathrm{tar}},\dot p_{\mathrm{tar}},
\phi_{\mathrm{tar}},\dot\phi_{\mathrm{tar}})=(0,0,0,0)$.
Define the learnable parameter vector
$\theta=(q_p,q_\phi,q_{\dot p},q_{\dot\phi},r_F)$ and
$Q_{\mathrm{cp}}(\theta)=\operatorname{diag}
(\theta_1,\theta_3,\theta_2,\theta_4)
=\operatorname{diag}(q_p,q_{\dot p},q_\phi,q_{\dot\phi})$ and
$\|z\|_{Q_{\mathrm{cp}}}^2=z^\top Q_{\mathrm{cp}}z$. The OCP is
\begin{subequations}
	\label{eq:app_cartpole_ocp}
	\begin{align}
		\min_{F_{0:N-1}}\quad
		&
		\begin{aligned}[t]
			&\sum_{k=0}^{N-1}
			\left(\|x_k-x_{\mathrm{tar}}\|_{Q_{\mathrm{cp}}}^2
			+r_FF_k^2\right) +10\|x_N-x_{\mathrm{tar}}\|_{Q_{\mathrm{cp}}}^2
		\end{aligned}
		\label{eq:app_cartpole_cost}
		\\
		\mathrm{s.t.}\quad
		&
		x_0=\xi,
		\label{eq:app_cartpole_initial}
		\\
		&
		x_{k+1}
		=
		f_{\mathrm{cp}}(x_k,F_k),
		\qquad
		k=0,\ldots,N-1,
		\label{eq:app_cartpole_dynamic}
		\\
		&
		-2\le F_k\le6,
		\qquad
		k=0,\ldots,N-1,
		\label{eq:app_cartpole_input}
		\\
		&
		\sum_{k=0}^{N-1}F_k^2\le70.
		\label{eq:app_cartpole_energy}
	\end{align}
\end{subequations}

Here, $f_{\mathrm{cp}}$ denotes the discrete dynamics defined in
Eq.~(\ref{eq:app_cartpole_dynamics}), with $m_c=1.0$, $m_p=0.1$,
$l=0.5$, $g=9.81$, and $\Delta t=0.05$. 
Here, $q_p$, $q_\phi$, $q_{\dot p}$, and $q_{\dot\phi}$ respectively
weight the cart position, pole angle, cart velocity, and pole angular-velocity
tracking errors, and $r_F$ penalizes the control effort.

For open-loop imitation, multiple initial states are sampled to evaluate the
robustness of the learned solution. The initial state
$\xi=(p_0,\dot p_0,\phi_0,\dot\phi_0)$ is sampled according to
\begin{equation}
	\begin{aligned}
		p_0
		&\sim
		\operatorname{Unif}[-0.5,0.5],
		&
		\dot p_0
		&\sim
		\operatorname{Unif}[-0.5,0.5],
		\\
		\phi_0
		&\sim
		\operatorname{Unif}[-\pi,\pi],
		&
		\dot\phi_0
		&\sim
		\operatorname{Unif}[-1,1].
	\end{aligned}
	\label{eq:app_cartpole_initial_sampling}
\end{equation}

\noindent\textbf{Planar quadrotor.}
For the planar quadrotor task, the system state is defined as
$x_k=(p_{x,k},p_{z,k},v_{x,k},v_{z,k},\alpha_k,\omega_k)$, where
$p_{x,k}$ and $p_{z,k}$ denote the horizontal and vertical positions,
$v_{x,k}$ and $v_{z,k}$ denote the corresponding velocities, and
$\alpha_k$ and $\omega_k$ denote the attitude angle and angular velocity,
respectively. The control input is $u_k=(T_k,\tau_k)$, where $T_k$ is the
collective thrust and $\tau_k$ is the rotational control input. The discrete
dynamics are given by
\begin{equation}
	\begin{aligned}
		v_{x,k+1}
		&=
		v_{x,k}
		-
		\Delta t\,T_k\sin\alpha_k,
		\\
		v_{z,k+1}
		&=
		v_{z,k}
		+
		\Delta t
		\left(
		T_k\cos\alpha_k-g
		\right),
		\\
		\omega_{k+1}
		&=
		\omega_k
		+
		4\Delta t\,\tau_k,
		\\
		\alpha_{k+1}
		&=
		\alpha_k
		+
		\Delta t\,\omega_{k+1},
		\\
		p_{x,k+1}
		&=
		p_{x,k}
		+
		\Delta t\,v_{x,k+1},
		\\
		p_{z,k+1}
		&=
		p_{z,k}
		+
		\Delta t\,v_{z,k+1}.
	\end{aligned}
	\label{eq:app_quadrotor_dynamics}
\end{equation}
We consider a finite-horizon planning problem that steers the quadrotor from
the initial state $\xi$ toward the target state
$x_{\mathrm{tar}}=(1.0,1.2,0,0,0,0)$. Define
$Q_{\mathrm{quad}}=\operatorname{diag}(q_p,q_p,q_v,q_v,q_\alpha,q_\omega)$,
$R_{\mathrm{quad}}=\operatorname{diag}(r_T,r_\tau)$, and
$u_{\mathrm{hov}}=(g,0)$. The OCP is
\begin{subequations}
	\label{eq:app_quadrotor_ocp}
	\begin{align}
		\min_{T_{0:N-1},\,\tau_{0:N-1}}\quad
		&
		\begin{aligned}[t]
			&\sum_{k=0}^{N-1}
			\left(
			\|x_k-x_{\mathrm{tar}}\|_{Q_{\mathrm{quad}}}^2
			+\|u_k-u_{\mathrm{hov}}\|_{R_{\mathrm{quad}}}^2
			\right)
			\\[-1mm]
			&\quad+15\|x_N-x_{\mathrm{tar}}\|_{Q_{\mathrm{quad}}}^2
		\end{aligned}
		\label{eq:app_quadrotor_cost}
		\\
		\mathrm{s.t.}\quad
		&
		x_0=\xi,
		\label{eq:app_quadrotor_initial}
		\\
		&
		x_{k+1}
		=
		f_{\mathrm{quad}}(x_k,T_k,\tau_k),
		\qquad
		k=0,\ldots,N-1,
		\label{eq:app_quadrotor_dynamic}
		\\
		&
		0\le T_k\le2.2g,
		\qquad
		-3\le\tau_k\le3,
		\qquad
		k=0,\ldots,N-1,
		\label{eq:app_quadrotor_input}
		\\
		&
		0.43\le p_{z,k}\le1.20,
		\qquad
		p_{x,k}\le1.0,
		\qquad
		k=0,\ldots,N,
		\label{eq:app_quadrotor_position}
		\\
		&
		|\alpha_k|\le0.34,
		\qquad
		k=0,\ldots,N.
		\label{eq:app_quadrotor_attitude}
	\end{align}
\end{subequations}
Here, $f_{\mathrm{quad}}$ denotes the discrete dynamics defined in
Eq.~(\ref{eq:app_quadrotor_dynamics}), with $g=9.81$ and $\Delta t=0.05$.
The thrust
regularization term is centered at $T_k=g$, corresponding to the nominal
hovering thrust of the normalized model.
The learnable parameter vector is
$\theta=(q_p,q_v,q_\alpha,q_\omega,r_T,r_\tau)$, where $q_p$,
$q_v$, $q_\alpha$, and $q_\omega$ respectively weight the position, velocity,
attitude-angle, and angular-velocity tracking errors, while $r_T$ and $r_\tau$
penalize the thrust and rotational control inputs.

For open-loop imitation, the initial state is sampled around
$\bar x_0=(-1.0,0.65,0,0,0.15,0):$
\begin{equation}
	\begin{aligned}
		p_{x,0}
		&=
		\bar p_{x,0}
		+
		\operatorname{Unif}[-0.25,0.25],
		&
		p_{z,0}
		&=
		\bar p_{z,0}
		+
		\operatorname{Unif}[-0.10,0.10],
		\\
		v_{x,0}
		&=
		\bar v_{x,0}
		+
		\operatorname{Unif}[-0.10,0.10],
		&
		v_{z,0}
		&=
		\bar v_{z,0}
		+
		\operatorname{Unif}[-0.10,0.10],
		\\
		\alpha_0
		&=
		\bar\alpha_0
		+
		\operatorname{Unif}[-0.08,0.08],
		&
		\omega_0
		&=
		\bar\omega_0
		+
		\operatorname{Unif}[-0.10,0.10].
	\end{aligned}
	\label{eq:app_quadrotor_initial_sampling}
\end{equation}
\noindent\textbf{Two-link robot arm.}
For the two-link robot-arm task, the system state is defined as
$q_k=(q_{1,k},q_{2,k})$, where $q_{1,k}$ and $q_{2,k}$ denote the two
joint angles. The control input is
$u_k=(\dot q_{1,k},\dot q_{2,k})$, where $\dot q_{1,k}$ and
$\dot q_{2,k}$ are the commanded joint velocities. The discrete joint
dynamics and the corresponding end-effector position are given by
\begin{equation}
	\begin{aligned}
		q_{k+1}
		&=
		q_k+\Delta t\,u_k,
		\\
		p_{\mathrm{ee}}(q_k)
		&=
		\begin{bmatrix}
			\cos q_{1,k}
			+
			0.8\cos(q_{1,k}+q_{2,k})
			\\
			\sin q_{1,k}
			+
			0.8\sin(q_{1,k}+q_{2,k})
		\end{bmatrix}.
	\end{aligned}
	\label{eq:app_robot_arm_dynamics}
\end{equation}

We consider a finite-horizon planning problem that steers the robot arm from
an initial joint configuration $\xi$ toward the target configuration
$q_{\mathrm{tar}}=(0.75,-0.65)$ while avoiding a circular obstacle in the
end-effector workspace. Define
\begin{equation*}
	\ell_{\mathrm{arm}}(q)
	=
	q_q\|q-q_{\mathrm{tar}}\|_2^2
	+q_{\mathrm{ee}}\|p_{\mathrm{ee}}(q)-p_{\mathrm{ee}}(q_{\mathrm{tar}})\|_2^2,
	\qquad
	R_{\mathrm{arm}}=\operatorname{diag}(r_1,r_2).
\end{equation*}
The OCP is
\vspace{-2mm}
\begin{subequations}
	\label{eq:app_robot_arm_ocp}
	\begin{align}
		\min_{u_{0:N-1}}\quad
		&\sum_{k=0}^{N-1}
		\left(\ell_{\mathrm{arm}}(q_k)+\|u_k\|_{R_{\mathrm{arm}}}^2\right)
		+20\ell_{\mathrm{arm}}(q_N)
		\label{eq:app_robot_arm_cost}
		\\
		\mathrm{s.t.}\quad
		&
		q_0=\xi,
		\label{eq:app_robot_arm_initial}
		\\
		&
		q_{k+1}
		=
		q_k+\Delta t\,u_k,
		\qquad
		k=0,\ldots,N-1,
		\label{eq:app_robot_arm_dynamic}
		\\
		&
		-3
		\le
		\dot q_{1,k}
		\le
		3,
		\qquad
		-3
		\le
		\dot q_{2,k}
		\le
		3,
		\qquad
		k=0,\ldots,N-1,
		\label{eq:app_robot_arm_input}
		\\
		&
		-1.10
		\le
		q_{1,k}
		\le
		0.75,
		\qquad
		-0.65
		\le
		q_{2,k}
		\le
		1.28,
		\qquad
		k=1,\ldots,N,
		\label{eq:app_robot_arm_joint}
		\\
		&
		(0.29+m_{\mathrm{obs}})^2
		-
		\left\|
		p_{\mathrm{ee}}(q_k)
		-
		\begin{bmatrix}
			1.30\\
			0.30
		\end{bmatrix}
		\right\|_2^2
		\le
		0,
		\qquad
		k=1,\ldots,N.
		\label{eq:app_robot_arm_obstacle}
	\end{align}
\end{subequations}
Here, the dynamics are defined in
Eq.~(\ref{eq:app_robot_arm_dynamics}), with $\Delta t=0.05$. The obstacle is centered at
$p_{\mathrm{obs}}=(1.30,0.30)$, and the constraint in
Eq.~(\ref{eq:app_robot_arm_obstacle}) requires the end effector to remain outside
a circle with radius $0.29+m_{\mathrm{obs}}$.
The learnable parameter vector is
$\theta=(q_q,q_{\mathrm{ee}},r_1,r_2,m_{\mathrm{obs}})$, where $q_q$
weights the joint-configuration tracking error, $q_{\mathrm{ee}}$ weights the
end-effector tracking error, $r_1$ and $r_2$ penalize the two commanded joint
velocities, and $m_{\mathrm{obs}}$ represents the learnable obstacle-safety
margin.

For open-loop imitation, the initial joint configuration is sampled according to
\begin{equation}
	q_{1,0}
	\sim
	\operatorname{Unif}[-1.10,-0.80],
	\qquad
	q_{2,0}
	\sim
	\operatorname{Unif}[1.05,1.28].
	\label{eq:app_robot_arm_initial_sampling}
\end{equation}

\textbf{Open-loop and closed-loop imitation protocols.}

\emph{Open-loop imitation.}
For each task, we sample \(N_{\mathrm{demo}}=32\) initial conditions
\(\{\xi_j\}_{j=1}^{N_{\mathrm{demo}}}\) according to the task-specific distributions in
Eqs.~(\ref{eq:app_cartpole_initial_sampling}),
(\ref{eq:app_quadrotor_initial_sampling}), and
(\ref{eq:app_robot_arm_initial_sampling}). Starting from each \(\xi_j\), the
OCP is first solved using the teacher parameter
\(\theta_{\mathrm{true}}\), producing the demonstration trajectories
\((X_j^{\mathrm{demo}},U_j^{\mathrm{demo}})\). During learning, the same OCP
is repeatedly solved from the same initial condition using the current
parameter \(\theta\), yielding
\((X^\star(\theta;\xi_j),U^\star(\theta;\xi_j))\). The parameter is then
updated by minimizing the imitation objective in
Eq.~(\ref{eq:imit_loss}) over all \(32\) initial conditions.

\emph{Closed-loop imitation.}
For closed-loop imitation, a reference rollout is first generated using
\(\theta_{\mathrm{true}}\). At each MPC instant, the OCP is solved from the
current teacher state, only the first element of the optimized control
trajectory is applied, and the system is propagated to the next state. This
procedure is repeated for \(50\) MPC steps, producing a closed-loop reference trajectory and the corresponding sequence of optimal predictions
\(\{(X_t^{\mathrm{demo}},U_t^{\mathrm{demo}})\}_{t=0}^{49}\).
The learner performs the same receding-horizon rollout using the current
parameter \(\theta\), producing
$
\left\{
\left(
X_t^\star(\theta),
U_t^\star(\theta)
\right)
\right\}_{t=0}^{49}.
$
At every MPC instant, the discrepancy between the learner and teacher
state-control predictions is evaluated using the same form as
Eq.~(\ref{eq:imit_loss}), and the losses are accumulated over the complete
rollout.

\noindent\textbf{Hyperparameter settings.}
Table~\ref{tab:ocp_training_hyperparameters} summarizes the learning and
\solver{lapanda} solver configurations. Here,
$\varepsilon_{\mathrm{in}}$ and $\varepsilon_{\mathrm{ALM}}$ are the PANDA
stationarity and ALM feasibility tolerances, while $\rho_0$ and $\tau_\rho$
are the initial penalty and its update factor. The matrix-free backward pass
uses the forward-converged penalty directly and solves the reduced system by
CG, with an automatic MINRES fallback if CG detects nonpositive curvature.

\vspace{-4mm}
\begin{table}[h]
	\centering
	\caption{
		Learning configurations and lapanda hyperparameters for
		the OCP imitation experiments.
	}
	\label{tab:ocp_training_hyperparameters}
	\setlength{\tabcolsep}{4.5pt}
	\renewcommand{\arraystretch}{1.12}
	\resizebox{\linewidth}{!}{%
		\begin{tabular}{llcccccc}
			\toprule
			Protocol
			& OCP
			& Epochs
			& Learning rate
			& $\varepsilon_{\mathrm{in}}/
			\varepsilon_{\mathrm{ALM}}$
			& $(\rho_0,\tau_\rho)$
			& Max iter.\ (PANDA/ALM)
			& Warm start
			\\
			\midrule
			
			\multirow{3}{*}{Open-loop}
			& CartPole
			& 800
			& $2\times10^{-3}$
			& $10^{-3}/10^{-3}$
			& $(10,5)$
			& $1500/20$
			& Previous epoch
			\\
			
			& Quadrotor
			& 800
			& $2\times10^{-3}$
			& $10^{-3}/10^{-3}$
			& $(10,5)$
			& $1500/20$
			& Previous epoch
			\\
			
			& Robot Arm
			& 800
			& $2\times10^{-3}$
			& $10^{-3}/10^{-3}$
			& $(10^3,5)$
			& $1500/20$
			& Previous epoch
			\\
			
			\midrule
			
			\multirow{3}{*}{Closed-loop}
			& CartPole
			& 500
			& $1\times10^{-3}$
			& $10^{-3}/10^{-3}$
			& $(10,5)$
			& $1500/20$
			& Previous MPC step
			\\
			
			& Quadrotor
			& 500
			& $1\times10^{-3}$
			& $10^{-3}/10^{-3}$
			& $(10,5)$
			& $1500/20$
			& Previous MPC step
			\\
			
			& Robot Arm
			& 500
			& $1\times10^{-3}$
			& $10^{-3}/10^{-3}$
			& $(10^3,5)$
			& $1500/20$
			& Previous MPC step
			\\
			
			\bottomrule
		\end{tabular}
	}
\end{table}

\vspace{-2mm}
The other differentiable OCP solvers use the same prediction horizon, sampling
interval, learning schedule, and stopping tolerance as the corresponding
protocol in Table~\ref{tab:ocp_training_hyperparameters}.
\vspace{-4mm}

\begin{table}[h]
	\centering
	\caption{
		Method-specific solver configurations of the differentiable OCP
		baselines.
	}
	\label{tab:ocp_baseline_solver_settings}
	\setlength{\tabcolsep}{5pt}
	\renewcommand{\arraystretch}{1.18}
	\resizebox{\linewidth}{!}{%
		\begin{tabular}{llll}
			\toprule
			Method
			& Forward pass
			& Backward pass
			& Maximum iterations
			\\
			\midrule
			
			SafePDP
			& IPOPT-based constrained OCP solve
			& Constrained auxiliary system (COC)
			& IPOPT: $3000$
			\\
			
			TurboMPC-CPU
			& SQP-ADMM with
			\texttt{admm\_jax\_loop\_pcg}
			& \texttt{admm\_jax\_loop\_pcg}
			& SQP: $50$; ADMM: $1000$
			\\
			
			TurboMPC-GPU
			& SQP-ADMM with
			\texttt{admm\_fused\_cudss}
			& \texttt{direct\_cudss\_ffi}
			& SQP: $50$; ADMM: $1000$
			\\
			
			\bottomrule
		\end{tabular}
	}
	\vspace{-2mm}
\end{table}

For \solver{SafePDP}, we evaluate its constrained auxiliary-system sensitivity
mode (\texttt{coc}) and barrier approximation with
$\gamma_{\mathrm{bar}}=10^{-2}$. The barrier variant exhibits an abrupt
initial change and a less stable imitation-loss trajectory, as shown in
Fig.~\ref{fig:safepdp_barrier_loss}; hence, the main paper reports the
\texttt{coc} results.

\begin{figure}[h]
	\centering
	\includegraphics[
	width=\linewidth,
	trim={0 0 0 0},
	clip
	]{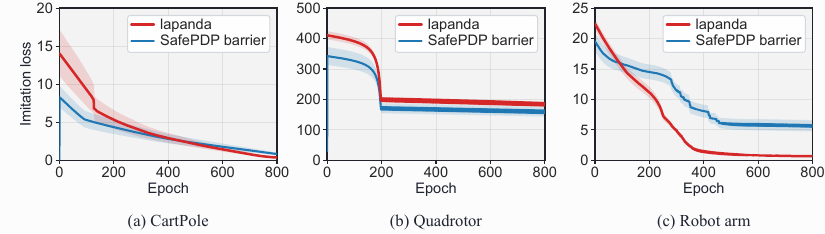}
	\vspace{-6mm}
	\caption{
		Imitation loss comparison between SafePDP-barrier and lapanda.
	}
	\vspace{-2mm}
	\label{fig:safepdp_barrier_loss}
\end{figure}

For \solver{TurboMPC}, we report the warm-started configuration
for each hardware backend as \solver{TurboMPC-CPU} and
\solver{TurboMPC-GPU}. Both use the same SQP and ADMM iteration limits and the
full Hessian, differing only in their CPU- and GPU-specific linear-algebra
backends.

\noindent\textbf{Sensitivity accuracy over learning.}
Table~\ref{tab:ocp_sensitivity_accuracy} evaluates the backward approximation
at representative learning checkpoints. At each sampled problem, we keep the
\solver{lapanda} forward solution fixed and use a high-accuracy solve of the
original KKT sensitivity system at the same point as the reference. The
reported metrics are computed from the gradients averaged over the \(32\)
open-loop scenarios or the \(50\) MPC instants of the closed-loop rollout.

\vspace{-4mm}
\begin{table}[h]
	\centering
	\caption{Sensitivity accuracy of \solver{lapanda} over the OCP imitation-learning experiments.}
	\label{tab:ocp_sensitivity_accuracy}
	\setlength{\tabcolsep}{3.2pt}
	\renewcommand{\arraystretch}{1.08}
	\resizebox{\linewidth}{!}{%
		\begin{tabular}{llcccccccc}
			\toprule
			\multirow{2}{*}{OCP}
			& \multirow{2}{*}{Metric}
			& \multicolumn{4}{c}{Open-loop epoch}
			& \multicolumn{4}{c}{Closed-loop epoch} \\
			\cmidrule(lr){3-6}
			\cmidrule(lr){7-10}
			& & \(0\) & \(200\) & \(400\) & \(800\)
			& \(0\) & \(125\) & \(250\) & \(500\) \\
			\midrule
			\multirow{2}{*}{CartPole}
			& Rel. err. (\%)
			& \(0.012\) & \(0.004\) & \(0.012\) & \(0.010\)
			& \(0.002\) & \(0.006\) & \(0.006\) & \(0.007\) \\
			& Cos. sim.
			& \(0.99999999\) & \(1.00000000\) & \(1.00000000\) & \(1.00000000\)
			& \(1.00000000\) & \(1.00000000\) & \(1.00000000\) & \(1.00000000\) \\
			\midrule
			\multirow{2}{*}{Quadrotor}
			& Rel. err. (\%)
			& \(0.048\) & \(0.286\) & \(0.363\) & \(1.306\)
			& \(0.058\) & \(0.137\) & \(0.134\) & \(0.076\) \\
			& Cos. sim.
			& \(0.99999997\) & \(0.99999999\) & \(0.99999998\) & \(0.99999982\)
			& \(0.99999992\) & \(0.99999999\) & \(1.00000000\) & \(1.00000000\) \\
			\midrule
			\multirow{2}{*}{Robot Arm}
			& Rel. err. (\%)
			& \(8.459\) & \(0.076\) & \(0.688\) & \(0.027\)
			& \(1.178\) & \(4.572\) & \(0.084\) & \(0.136\) \\
			& Cos. sim.
			& \(0.99643301\) & \(0.99999972\) & \(0.99999237\) & \(1.00000000\)
			& \(0.99993896\) & \(0.99901230\) & \(0.99999970\) & \(0.99999995\) \\
			\bottomrule
		\end{tabular}%
	}
\end{table}

Across the learning checkpoints, the relative sensitivity error is generally
small, with larger deviations confined to a few Robot Arm checkpoints with
less favorable numerical conditioning. More importantly, the cosine similarity
remains consistently above \(0.996\), indicating that the approximate gradients
preserve the reference gradient directions even when their magnitudes are less
accurate. This strong directional agreement helps explain why, in
Figs.~\ref{fig:ocp_learning_results} and
\ref{fig:closed_loop_learning_results}, \solver{lapanda} produces loss
trajectories that are nearly identical to those of the corresponding baselines,
even when using the balanced sensitivity strategy.

\noindent\textbf{Timing and memory measurement.}
For the runtime comparison, all forward and backward computation times are arithmetic means of wall-clock time over all recorded solves. For \solver{lapanda} and \solver{SafePDP}, the forward and backward procedures
can be timed separately. In contrast, \solver{TurboMPC} does not expose an
independent backward call. We therefore first
measure a forward-only solve and then measure the complete value-and-gradient
evaluation. Its backward time is estimated as
\begin{equation}
	t_{\mathrm{bwd}}^{\mathrm{TurboMPC}}
	=
	t_{\mathrm{value+grad}}^{\mathrm{TurboMPC}}
	-
	t_{\mathrm{fwd}}^{\mathrm{TurboMPC}}.
	\label{eq:app_turbompc_backward_time}
\end{equation}
In addition to the averaged runtime results, we record the forward and backward
times at each MPC instant during one representative closed-loop rollout, as
shown in Fig.~\ref{fig:closed_loop_mpc_timing}. The solvers exhibit different
runtime profiles over the rollout. In particular, the computation time of
\solver{lapanda} generally decreases after the initial MPC steps. This trend is consistent with the benefit of warm-starting: the primal variables, multipliers, penalty parameters, and shifted control trajectory obtained at the previous MPC instant provide an increasingly accurate initialization for
the subsequent problem.

\begin{figure}[h]
	\centering
	\includegraphics[
	width=\linewidth,
	trim={0 2mm 0 2mm},
	clip
	]{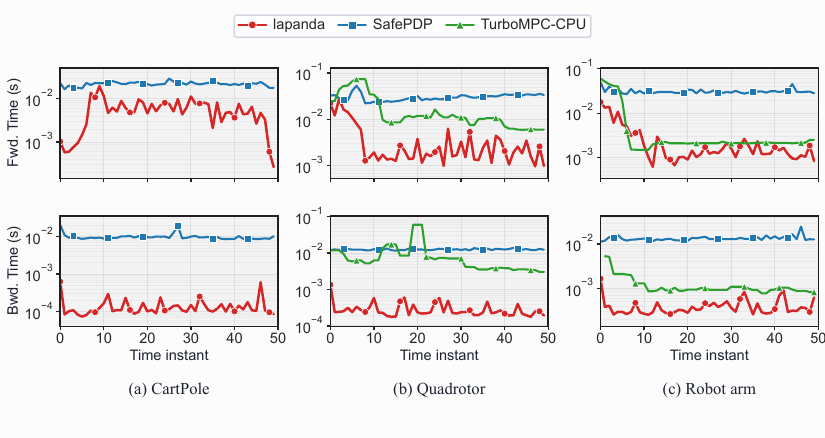}
	\vspace{-10mm}
	\caption{
		Computation times at each MPC instant during a representative closed-loop rollout.
	}
	\label{fig:closed_loop_mpc_timing}
\end{figure}

\begin{wrapfigure}{r}{0.5\linewidth}
	\centering
	\includegraphics[
	width=0.95\linewidth,
	trim={4mm 0 0 0},
	clip
	]{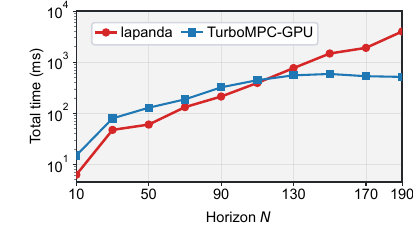}
	\vspace{-8mm}
	\caption{Mean solve time versus horizon length.}
	\label{fig:ocp_horizon_scaling}
	\vspace{-6mm}
\end{wrapfigure}
\noindent\textbf{Scaling with the horizon length.}
To evaluate solver performance across different horizon lengths $N$, we vary
$N$ up to $200$ for the Quadrotor OCP while fixing the total
prediction time at $2.5$ s. Thus, different values of $N$ correspond to
different temporal discretization resolutions of the same prediction window.
For each $N$, we report the mean total forward-and-backward time over a
30-step warm-started MPC rollout. We use \solver{TurboMPC-GPU} as a
representative GPU-based baseline.

As shown in Fig.~\ref{fig:ocp_horizon_scaling}, \solver{lapanda} is faster for smaller $N$, whereas \solver{TurboMPC-GPU} becomes more advantageous as $N$ increases. This result clarifies the intended application regime of our method: \solver{lapanda} is particularly suitable for solving complex short-horizon problems on resource-constrained platforms, while specialized GPU solvers are preferable for large-horizon problems when dedicated GPU resources are available.

\subsection{Embedded Obstacle-Avoidance Experiments}
\label{app:embedded_experiment_details}

\noindent\textbf{Embedded platform.}
The embedded experiments are conducted on an NVIDIA Jetson Orin Nano platform
equipped with a six-core Arm Cortex-A78AE v8.2 CPU, a 1024-core NVIDIA Ampere
GPU with 32 Tensor Cores, 8\,GB of 128-bit LPDDR5 memory. The reported \solver{lapanda} experiments execute the exported
standalone C solver on the Arm CPU.

\noindent\textbf{Circular-obstacle task.}
We first consider a smooth nonlinear obstacle-avoidance problem. The vehicle is
described by a kinematic bicycle model with state
$x_k=(p_{x,k},p_{y,k},\psi_k)$ and control input
$u_k=(v_k,\delta_k)$, where $(p_{x,k},p_{y,k})$ is the planar position,
$\psi_k$ is the heading angle, $v_k$ is the longitudinal velocity, and
$\delta_k$ is the steering angle. The discrete dynamics are
\begin{equation}
	\begin{aligned}
		p_{x,k+1}
		&=
		p_{x,k}
		+
		\Delta t\,v_k\cos\psi_k,
		\\
		p_{y,k+1}
		&=
		p_{y,k}
		+
		\Delta t\,v_k\sin\psi_k,
		\\
		\psi_{k+1}
		&=
		\psi_k
		+
		\frac{\Delta t}{L}
		v_k\tan\delta_k,
	\end{aligned}
	\label{eq:app_embedded_vehicle_dynamics}
\end{equation}
where horizon $N=12$, $\Delta t=0.12$, and $L=0.45$.

Let
$p_k=(p_{x,k},p_{y,k})$ denote the vehicle position. The full differentiable
problem-parameter vector is
$\theta_{\mathrm{circ}}=(q_p,q_\psi,r_v,r_\delta,q_f,r_o,c_y)$, where
$c_{\mathrm{circ}}=(0,c_y)\in\mathbb R^2$ and
$r_{\mathrm{circ}}(\theta_{\mathrm{circ}})=r_o$ denote the center and radius
of the circular obstacle, respectively. The smooth obstacle constraint is
\begin{equation}
	h_k^{\mathrm{circ}}
	\left(
	x_k;\theta_{\mathrm{circ}}
	\right)
	:=
	r_{\mathrm{circ}}
	\left(
	\theta_{\mathrm{circ}}
	\right)^2
	-
	\left\|
	p_k-c_{\mathrm{circ}}
	\right\|_2^2
	\le0.
	\label{eq:app_circle_constraint}
\end{equation}

Starting from
$x_{\mathrm{init}}=(-1.2,0,0)$, the vehicle is required to reach
$x_{\mathrm{tar}}=(1.2,0,0)$ while remaining outside the circular obstacle.
The corresponding OCP is
\begin{subequations}
	\label{eq:app_circle_ocp}
	\begin{align}
		\min_{\substack{u_{0:N-1}}}\quad
		&
		\begin{aligned}[t]
			&
			\sum_{k=0}^{N-1}
			\Big[
			q_p
			\left\|
			p_{k+1}-p_{\mathrm{tar}}
			\right\|_2^2
			+
			q_\psi
			\left(
			\psi_{k+1}-\psi_{\mathrm{tar}}
			\right)^2
			+
			r_v v_k^2
			+
			r_\delta\delta_k^2
			\Big]
			\\
			&+
			q_f
			\left[
			q_p
			\left\|
			p_N-p_{\mathrm{tar}}
			\right\|_2^2
			+
			q_\psi
			\left(
			\psi_N-\psi_{\mathrm{tar}}
			\right)^2
			\right]
		\end{aligned}
		\label{eq:app_circle_cost}
		\\
		\mathrm{s.t.}\quad
		&
		x_0=x_{\mathrm{init}},
		\label{eq:app_circle_initial}
		\\
		&
		x_{k+1}
		=
		f_{\mathrm{veh}}(x_k,u_k),
		\qquad
		k=0,\ldots,N-1,
		\label{eq:app_circle_dynamic}
		\\
		&
		-1.5\le v_k\le1.5,
		\qquad
		-25^\circ\le\delta_k\le25^\circ,
		\qquad
		k=0,\ldots,N-1,
		\label{eq:app_circle_input}
		\\
		&
		h_k^{\mathrm{circ}}
		\left(
		x_k;\theta_{\mathrm{circ}}
		\right)
		\le0,
		\qquad
		k=1,\ldots,N.
		\label{eq:app_circle_avoidance}
	\end{align}
\end{subequations}

Here, we use
$\theta_{\mathrm{circ}}=(10.0,0.2,10^{-2},10^{-2},30.0,0.30,0.20)$ and $f_{\mathrm{veh}}$ denotes the dynamics in
Eq.~(\ref{eq:app_embedded_vehicle_dynamics}). Given a teacher control trajectory $U^{\mathrm{teacher}}$, the outer learning
objective is
\begin{equation}
	\mathcal L_{\mathrm{circ}}
	\left(
	\theta_{\mathrm{circ}}
	\right)
	=
	\frac{1}{2}
	\left\|
	U^\star
	\left(
	\theta_{\mathrm{circ}}
	\right)
	-
	U^{\mathrm{teacher}}
	\right\|_2^2.
	\label{eq:app_circle_imitation_loss}
\end{equation}
For the embedded timing benchmark, we set
$U^{\mathrm{teacher}}=0$; a nonzero teacher changes the adjoint right-hand
side but not the solver configuration.
Both \solver{lapanda} and \solver{acados} successfully solve the circular-obstacle task. Figure~\ref{fig:embedded_circle_plan} shows the resulting vehicle trajectories and the corresponding obstacle clearances.
\begin{figure}[h]
	\centering
	\includegraphics[
	width=\linewidth,
	trim={0 12mm 0 6mm},
	clip
	]{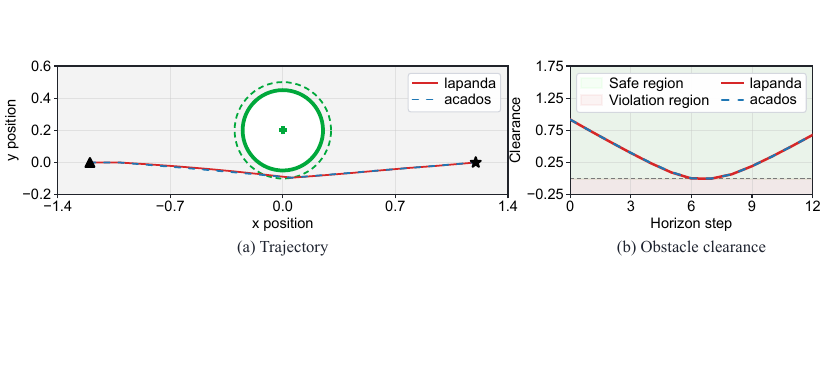}
	\vspace{-10mm}
	\caption{
		Planned trajectory for the circular-obstacle avoidance task on the
		embedded platform.
	}
	\vspace{-4mm}
	\label{fig:embedded_circle_plan}
\end{figure}

\medskip
\noindent\textbf{Rectangular-obstacle task.}
The rectangular task uses the same vehicle model, initial and target states,
and planning objective as in Eq.~(\ref{eq:app_circle_ocp}), but adopts the
horizon $N=20$, velocity bounds $-1\leq v_k\leq1$, and steering-angle bounds
$-0.7\leq\delta_k\leq0.7$~rad. For this task, the PANDA and ALM tolerances
are set to $10^{-3}$ and $10^{-4}$, respectively. Its obstacle constraint is
defined by
\begin{equation*}
	x_{\min}=-0.35-m_l,
	\qquad
	x_{\max}=0.35+m_r,
	\qquad
	y_{\min}=-0.22-m_b,
	\qquad
	y_{\max}=0.22+m_t.
\end{equation*}
A point lies outside the rectangle if at least one of the following conditions
holds:
\begin{equation}
	p_{x,k}\le x_{\min}
	\quad\vee\quad
	p_{x,k}\ge x_{\max}
	\quad\vee\quad
	p_{y,k}\le y_{\min}
	\quad\vee\quad
	p_{y,k}\ge y_{\max}.
	\label{eq:app_rectangle_disjunction}
\end{equation}
Thus, the feasible region is a nonconvex union of four half-spaces rather than
a single smooth inequality set. To encode this logical disjunction, define the
nonnegative penetration terms
\begin{equation*}
	\begin{aligned}
		s_{\mathrm{L},k}
		&=
		\max(p_{x,k}-x_{\min},0),
		&
		s_{\mathrm{R},k}
		&=
		\max(x_{\max}-p_{x,k},0),
		\\
		s_{\mathrm{B},k}
		&=
		\max(p_{y,k}-y_{\min},0),
		&
		s_{\mathrm{T},k}
		&=
		\max(y_{\max}-p_{y,k},0).
	\end{aligned}
\end{equation*}
When the vehicle is outside the rectangle, at least one of these terms is zero.
When it is strictly inside the rectangle, all four terms are positive.
Consequently, the disjunction in
Eq.~(\ref{eq:app_rectangle_disjunction}) can be represented by the
complementarity-style constraint
\begin{equation}
	\chi_{\mathrm{rect}}(x_k;\vartheta_{\mathrm{rect}})
	:=
	\frac{1}{2}
	s_{\mathrm{L},k}^2
	s_{\mathrm{R},k}^2
	s_{\mathrm{B},k}^2
	s_{\mathrm{T},k}^2
	=
	0.
	\label{eq:app_rectangle_mpcc_constraint}
\end{equation}
This product constraint is an MPCC-style representation of the logical
``or'' relation: it is zero whenever at least one separating condition is
satisfied and positive only when the vehicle lies inside the obstacle. We
therefore impose Eq.~(\ref{eq:app_rectangle_mpcc_constraint}) directly as a
nonlinear equality constraint.

The full problem-parameter vector supplied to the solver and its learnable
subset are
\begin{equation*}
	\vartheta_{\mathrm{rect}}
	=(q_p,q_\psi,r_v,r_\delta,q_f,m_l,m_r,m_b,m_t),
	\qquad
	\theta_{\mathrm{rect}}=(m_l,m_r,m_b).
\end{equation*}
We fix $(q_p,q_\psi,r_v,r_\delta,q_f)=(5.0,0.2,10^{-2},10^{-2},20.0)$
and $m_t=0.220$. Thus,
\begin{equation*}
	\vartheta_{\mathrm{rect}}(\theta_{\mathrm{rect}})
	=(5.0,0.2,10^{-2},10^{-2},20.0,m_l,m_r,m_b,0.220).
\end{equation*}
The teacher and initial full obstacle-margin vectors are
\begin{equation*}
	m^{\mathrm{teacher}}=(0.120,0.120,0.010,0.220),
	\qquad
	m^{\mathrm{init}}=(0.150,0.150,0.160,0.220).
\end{equation*}
The rectangular-obstacle imitation loss is
\begin{equation}
	\mathcal L_{\mathrm{rect}}
	\left(
	\theta_{\mathrm{rect}}
	\right)
	=
	\frac{1}{2}
	\left\|
	U^\star
	\left(
	\vartheta_{\mathrm{rect}}(\theta_{\mathrm{rect}})
	\right)
	-
	U^{\mathrm{teacher}}
	\right\|_2^2.
	\label{eq:app_rectangle_imitation_loss}
\end{equation}
\noindent\textbf{Embedded solver configuration.}
The principal solver settings used in the embedded experiments are summarized
in Table~\ref{tab:embedded_solver_settings}. The task definitions, horizon,
control bounds, obstacle geometry, and learnable parameters have been specified
above and are therefore omitted from the table.

\begin{table}[h]
	\centering
	\caption{
		Solver settings for the embedded obstacle-avoidance experiments.
	}
	\label{tab:embedded_solver_settings}
	\setlength{\tabcolsep}{3.8pt}
	\renewcommand{\arraystretch}{1.15}
	\resizebox{\linewidth}{!}{%
		\begin{tabular}{llcccc}
			\toprule
			Task
			& Method
			& Hessian treatment
			& Maximum iterations
			& Forward tol.
			& ALM settings
			\\
			\midrule
			
			\multirow{2}{*}{Circle}
			& lapanda
			& --
			& PANDA/ALM: $2000/100$
			& $10^{-1}/(2\times10^{-3})$
			& $\rho_0=10^4$, $\tau_\rho=10$
			\\
			
			& acados
			& EXACT + MIRROR
			& NLP/QP: $1000/200$
			& $2\times10^{-3}$
			& --
			\\
			
			\midrule
			
			\multirow{2}{*}{Rectangle}
			& lapanda
			& --
			& PANDA/ALM: $2000/100$
			& $10^{-3}/10^{-4}$
			& $\rho_0=10^4$, $\tau_\rho=10$
			\\
			
			& acados
			& Gauss-Newton + MIRROR
			& NLP/QP: $1000/50$
			& $10^{-4}$
			& --
			\\
			
			\bottomrule
		\end{tabular}
	}
\end{table}

For the rectangular task, we tested Gauss-Newton and exact-Hessian models in
\solver{acados}, enabled \texttt{MIRROR} regularization, increased the NLP
iteration limit to $1000$, and used a separate sensitivity-enabled solver.
Nevertheless, \solver{acados} repeatedly returned a failure status and the imitation loss remained nearly unchanged. This behavior is consistent with the local degeneracy of Eq.~(\ref{eq:app_rectangle_mpcc_constraint}). This experiment should thus be interpreted as an empirical stress test outside the regularity assumptions of our local theory.

\medskip
\noindent\textbf{Smoothed rectangular-obstacle benchmark.}
To complement the MPCC experiment and provide a more comprehensive evaluation of \solver{lapanda}, we additionally consider a smoothed version of the rectangular obstacle-avoidance problem. Specifically, we replace the disjunction by the conservative smooth maximum
\begin{equation}
	\widetilde\chi_{\tau_{\mathrm{sm}}}(p_k)
	=
	\tau_{\mathrm{sm}}\log\!\left(
	\frac{1}{4}\sum_{i=1}^{4}
	\exp\!\left(\frac{d_i(p_k)}{\tau_{\mathrm{sm}}}\right)
	\right)
	\ge 0,
	\label{eq:app_smoothed_rectangle_constraint}
\end{equation}
where
$
d(p_k)=
(x_{\min}-p_{x,k},\,
p_{x,k}-x_{\max},\,
y_{\min}-p_{y,k},\,
p_{y,k}-y_{\max}).
$
We use \(\tau_{\mathrm{sm}}=0.05\), which rounds the four corners while conservatively
preserving the obstacle. Since \solver{acados} still fails to complete the task
from a zero-control initialization, we manually construct a feasible nonzero
trajectory for its initial solve and shift the resulting solution in subsequent
MPC steps. As shown in Fig.~\ref{fig:app_smoothed_mpcc}, both
\solver{acados} and \solver{lapanda} successfully avoid the smoothed obstacle,
and their computation times become comparable during the closed-loop rollout.
In particular, the computation time of \solver{lapanda} progressively decreases
as the shifted warm start becomes effective and eventually stabilizes at a
similar millisecond-scale level. These results demonstrate the competitive computational performance of \solver{lapanda} on the smoothed obstacle-avoidance problem.

\vspace{-4mm}
\begin{figure}[h]
	\centering
	\includegraphics[
	width=\linewidth,
	trim={0 16mm 0 4mm},
	clip
	]{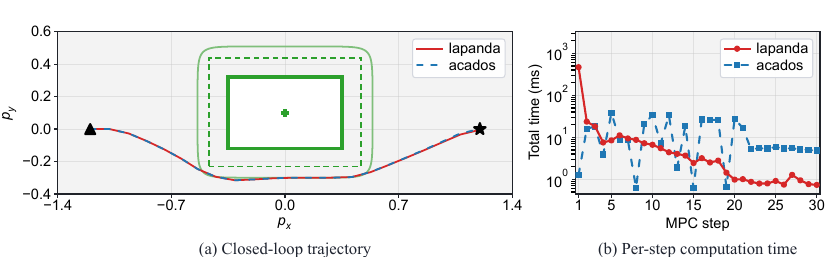}
	\vspace{-6mm}
	\caption{Solution trajectories and computation times for the smoothed rectangular-obstacle task.}
	\label{fig:app_smoothed_mpcc}
\end{figure}
\vspace{-4mm}
\clearpage
\section{Implementation Details}
\label{app:implementation}

This section briefly describes the software engineering design of
\solver{lapanda}, including its multi-platform implementation, computational
optimizations, and user-oriented interfaces.

\medskip
\begin{wrapfigure}{r}{0.52\linewidth}
	\centering
	\vspace{-5mm}
	\includegraphics[
	width=\linewidth,
	trim={0 0 0 0},
	clip
	]{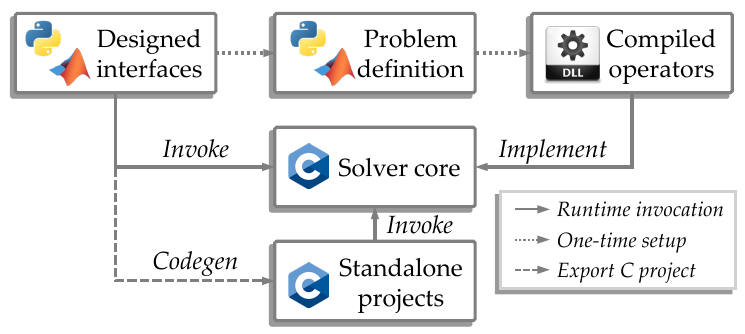}
	\renewcommand{\figurename}{Fig.}
	\vspace{-4mm}
	\caption{
		Software architecture and multi-platform workflow of
		lapanda.
	}
	\vspace{-2mm}
	\label{fig:implementation_architecture}
\end{wrapfigure}
\noindent\textbf{Multi-platform support.}
As illustrated in Fig.~\ref{fig:implementation_architecture}, the core
algorithms of \solver{lapanda} are implemented in C to efficiently execute the
computationally intensive optimization and sensitivity routines. Python and
MATLAB interfaces are provided through \solver{Pybind11} and \solver{MEX},
respectively. On both platforms, users define the parametric optimization
problem using \solver{CasADi}, while the interfaces generate the
problem-dependent operators and invoke the C backend for forward and backward
computation.

In addition, we provide an export utility on the Python platform that
automatically packages a specified problem into a minimal standalone CMake
project. This enables the same
optimization problem to be deployed on embedded C platforms.

\textbf{Engineering optimizations.}
The implementation incorporates two main engineering optimizations. First, once a solve is initiated, the ALM outer iterations, \solver{PANDA} inner iterations, matrix-free operator evaluations, and backward sensitivity iterations are carried out in the C backend. The high-level interfaces mainly handle problem definition, solver configuration, data transfer, and function invocation, thereby reducing repeated language-boundary communication.

Second, we reuse the problem-dependent operators across repeated solver calls. A considerable one-time cost arises from constructing the \solver{CasADi} computational graph and generating the corresponding objective, constraint, and derivative operators. Once the problem structure is fixed, subsequent solves require only evaluations of these pre-generated operators with updated states, parameters, and warm-start variables. We therefore compile the generated operators into a shared library that can be loaded and reused by the C solver, avoiding repeated graph construction and code generation during learning or MPC rollouts and reducing solver initialization time and memory overhead.

\noindent\textbf{User-oriented interfaces.}
The Python and MATLAB interfaces hide operator generation, compilation, and
C-backend configuration: users define the parametric problem and numerical
inputs, and the solver returns \(u^\star\) and its gradient with respect to
\(\theta\), while \texttt{variable} denotes inputs excluded from
differentiation. Representative workflows are shown below.

\begin{lstlisting}[
	style=almPython,
	caption={Python usage.},
	label={lst:python_interface}
	]
	# 1. Define a parametric optimization problem
	u = ca.SX.sym("u", n_u)
	theta = ca.SX.sym("theta", n_theta)
	variable = ca.SX.sym("variable", n_v)
	
	problem = CasadiProblem(
		u=u, theta=theta, variable=variable,
		cost=..., constraints=..., outer_loss=...)
	
	# 2. Generate operators and build the solver
	solver = build_solver(problem, name="demo_problem")
	
	# 3. Solve and differentiate
	result = solver.solve_alm(
		x0=x0, theta=theta_value, variable=variable_value,
		constraint_lower=..., constraint_upper=...)
		
	u_star = result["solution"]
	grad_theta = result["grad_theta"]
\end{lstlisting}

The MATLAB interface follows a similar workflow: the user defines the \solver{CasADi}
problem, creates the generated MEX solver, and invokes the forward and backward
computations.

\begin{lstlisting}[
	style=almMatlab,
	caption={MATLAB usage.},
	label={lst:matlab_interface}
	]
	% 1. Define a parametric optimization problem
	u = SX.sym('u', n_u, 1);
	theta = SX.sym('theta', n_theta, 1);
	variable = SX.sym('variable', n_v, 1);
	
	problem.u = u;
	problem.theta = theta;
	problem.variable = variable;
	problem.cost = ...;
	problem.constraints = ...;
	problem.outer_loss = ...;
	
	% 2. Generate operators and build the solver
	solver = lapanda_create_solver( ...
		problem, output_directory, 'demo_problem');
	
	% 3. Solve and differentiate
	result = solver.solve_alm( ...
		x0, theta_value, variable_value, ...
		constraint_lower, constraint_upper);
		
	u_star = result.solution;
	grad_theta = result.grad_theta;
\end{lstlisting}
\vspace{-2mm}
For embedded deployment, the Python interface can export a standalone CMake
project containing the required problem-dependent operators, solver source
files, and build scripts. Users can then add application-specific logic to the
exported project, compile it, and run the resulting executable.

\begin{lstlisting}[
	style=almPython,
	caption={Exporting a problem for standalone C deployment.},
	label={lst:c_export}
	]
	from lapanda import export_c_project
	export_c_project(
		problem,
		out_dir="exports/demo_problem",
		name="demo_problem",
	)
\end{lstlisting}
\vspace{-2mm}
\begin{lstlisting}[
	style=almC,
	caption={Standalone C usage.},
	label={lst:standalone_c}
	]
	#include "static_casadi_oracle.h"
	#include "lapanda_generated_config.h"
	int main(void) {
		alm_problem problem;
		double theta[LAPANDA_NTHETA] = {...}, variable[LAPANDA_NVAR] = {...};
		double u[LAPANDA_N] = {...}, grad_theta[LAPANDA_NTHETA];
		...
		
		// Initialize the generated operators with backward enabled.
		backward_params.enable = 1;
		lapanda_static_init_alm_problem(&problem, constraint_lower, constraint_upper,
			&solver_params, &backward_params);
			
		// The solution and gradient are returned in u and grad_theta.
		alm_solve_with_backward(&problem, &alm_params, u,
			multipliers, theta, variable, &info,
			grad_theta, &backward_info);
		return 0;
	}
\end{lstlisting}

\end{document}